\documentclass[11pt,english]{amsart} 

\usepackage[margin=1in]{geometry}
\usepackage{amsmath, amssymb, amsthm, float, tikz, tikz-cd, epstopdf}
\usepackage{amscd}

\usepackage[unicode,colorlinks,plainpages=false,hyperindex=true,bookmarksnumbered=true,bookmarksopen=false,pdfpagelabels]{hyperref}
\hypersetup{urlcolor=cyan,linkcolor=blue,citecolor=red,colorlinks=true}
\usepackage{color}
\usepackage{enumitem}

\usepackage{datetime} 

\numberwithin{equation}{section}
\newtheorem{thm}{Theorem}[section]
\newtheorem{prop}[thm]{Proposition}
\newtheorem{lem}[thm]{Lemma}
\newtheorem{cor}[thm]{Corollary}

\theoremstyle{definition}
\newtheorem{defn}[thm]{Definition}
\newtheorem{exa}[thm]{Example}
\newtheorem{cons}[thm]{Construction}
\newtheorem{nota}[thm]{Notation}

\newtheorem{conv}[thm]{Convention}

\theoremstyle{remark}
\newtheorem{rem}[thm]{Remark}

\newcommand{\Z}{\mathbb{Z}}

\newcommand{\lk}{\operatorname{lk}}

\title{Higher Alexander Polynomials of Satellite Links}
\date{\today}
\author{Soheil Azarpendar}
\address{Mathematical Institute, University of Oxford, Andrew Wiles Building, Radcliffe Observatory Quarter, Woodstock Road, Oxford, OX2 6GG, UK}
\email{azarpendar@maths.ox.ac.uk}
\author{Colin McCulloch}
\address{Mathematical Institute, University of Oxford, Andrew Wiles Building, Radcliffe Observatory Quarter, Woodstock Road, Oxford, OX2 6GG, UK}
\email{colin.mcculloch@maths.ox.ac.uk}

\begin{document}

\begin{abstract}
We present formulas for the higher Alexander polynomials of satellite links.
\end{abstract}

\maketitle

\section{Introduction}

\subsection{The satellite formula: history and methods}\hfill

Let $K\subset S^{3}$ be an oriented knot, called the \emph{companion}, and let
\[
P=P_{1}\cup\cdots\cup P_{m}\subset S^{1}\times D^{2}
\]
be an oriented link, called the \emph{pattern}. Note that the pattern $P$ may be
contained in a $3$-ball in $S^{1}\times D^{2}$. Let $N(K)$ be a tubular neighbourhood of $K$.
Given a framing $n\in\mathbb Z$, choose a diffeomorphism
\[
f_n\colon S^{1}\times D^{2}\longrightarrow N(K)
\]
realising this framing. The link $f_n(P)\subset S^{3}$ is called the
\emph{$n$-twisted satellite} of $K$ with pattern $P$, and is denoted
by $K(P,n)$.

\begin{figure}[htbp]
\centering
\includegraphics[scale=0.2]{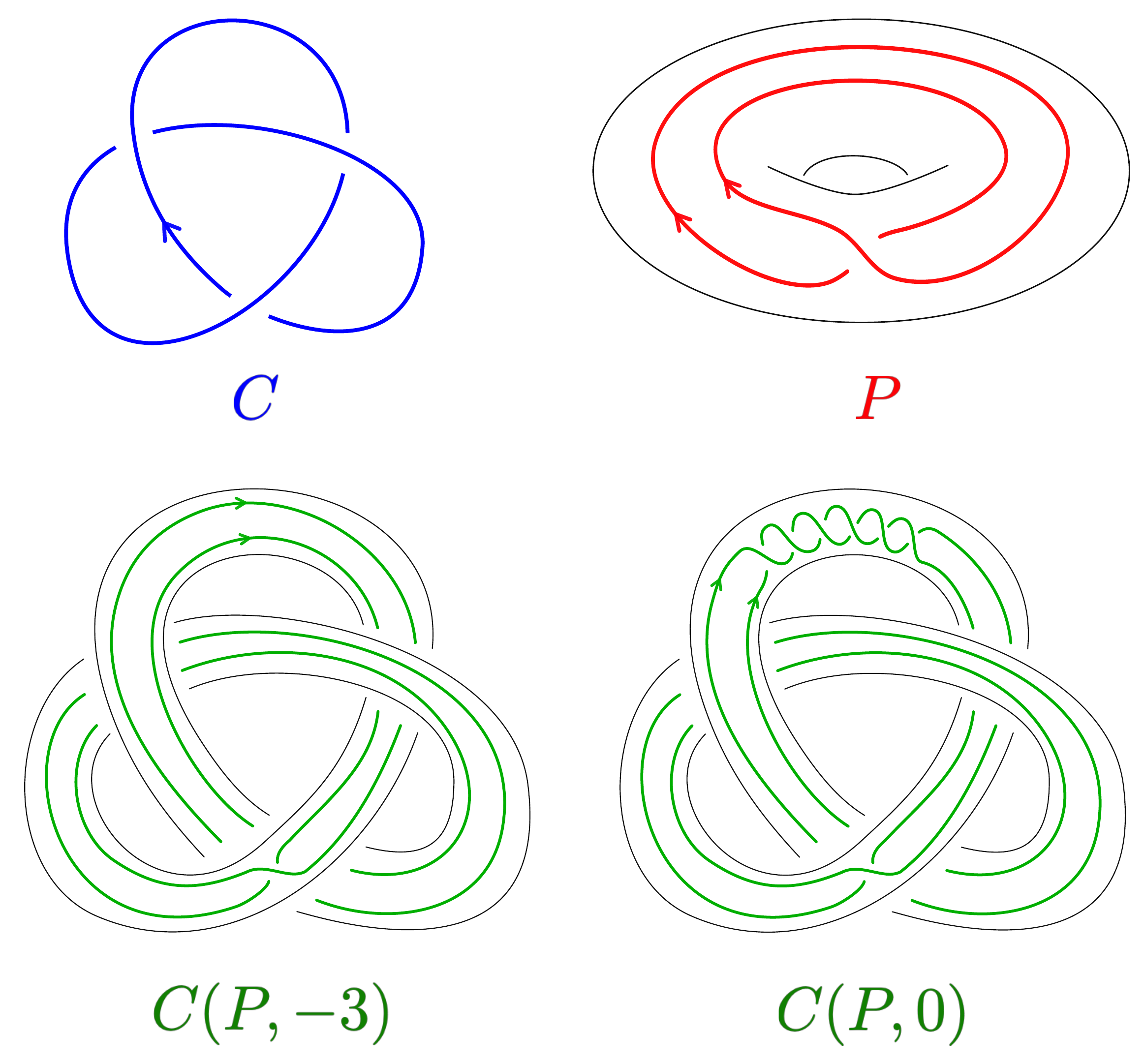}
\caption{Two satellites of the trefoil knot with $(2,1)$-torus knot pattern. The second example is the $(2,1)$-cable of the trefoil.}
\label{Fig:SatelliteExample}
\end{figure}
The behaviour of knot and link invariants under satellite operations has been studied extensively. One of the classical results in this direction describes the Alexander polynomial of a satellite in terms of the Alexander polynomials of its companion and pattern. In the special case where the pattern is a knot, the satellite formula goes back to Seifert~\cite{Seifert1950}.

Recall the winding number of $P$ is the class
\[
[P]\in H_1(S^1\times D^2)\cong \Z.
\]
Note that there are two choices for the identification here. In all the cases we consider the identification comes from the orientation of the core of the tubular neighbourhood.
\begin{thm}\label{thm:Seifert-formula}
Let $P\subset S^{1}\times D^{2}$ be a knot of winding number $w$. Then
\[
\Delta_{K(P,0)}(t)
\doteq
\Delta_K(t^{w})\Delta_P(t).
\]
\end{thm}

\begin{rem}
In the formula above, the pattern $P$ is regarded as a knot in $S^{3}$
by embedding $S^{1}\times D^{2}$ as a standard unknotted solid torus.
Thus its Alexander polynomial
$\Delta_P(t)$ is well-defined up to multiplication
by a unit. We use this convention throughout the paper
whenever referring to the Alexander polynomial of a pattern. 
\end{rem}

In his influential paper, Torres~\cite{Torres} extended this formula to the case of a pattern with several components, obtaining what is now known as the Seifert--Torres formula.

\begin{thm}
For each $i$, let $w_i$ denote the winding number of $P_i$ in the solid torus. Then
\[
\Delta_{K(P,0)}(t_{1},\ldots,t_{m})
\doteq
\Delta_K\!\left(t_{1}^{w_{1}}\cdots t_{m}^{w_{m}}\right)
\Delta_P(t_{1},\ldots,t_{m}).
\]
\end{thm}

Since then, the formula has been reproven and extended in several directions, and the methods used vary widely. Torres's original proof relies on a detailed Fox-calculus computation. Later, Lickorish~\cite{Lickorish} gave a more geometric proof by constructing a Seifert surface for the satellite from suitable Seifert surfaces for the companion and pattern. Petkova~\cite{Petkova} obtained the same formula by decategorifying the gluing theorem in bordered Floer homology.

A different line of proof lifts the satellite decomposition to the maximal abelian cover. The first homology of this cover is the Alexander module, and it turns out that in case the pattern is a knot the formula can be recovered by a Mayer--Vietoris argument. 

Livingston and Melvin~\cite{Livingston} point out that this viewpoint is implicit in Seifert's work and was later recovered by Shinohara. They also derive a formula for the Alexander module itself, from which the Seifert--Torres formula follows. A particularly clear exposition of this argument appears in the book of Burde, Zieschang, and Heusener~\cite{Burde}.

In this paper, we adapt this module-theoretic approach to
extend the satellite formula to the setting in which both the
companion and the pattern are allowed to be links. The general setting is explained in the following.

\begin{cons}[Satellite construction for links]
Let
\[
C=C_{1}\cup\cdots\cup C_{\mu}\subset S^{3}
\]
be an oriented link. For $1\leq i,j\leq\mu$, set
\[
l_{ij}=
\begin{cases}
\lk(C_i,C_j), & i\neq j,\\
0, & i=j.
\end{cases}
\]
For each $i\in\{1,\ldots,\mu\}$, let
\[
P_i=P_{i1}\cup\cdots\cup P_{im_i}\subset S^{1}\times D^{2}
\]
be an oriented link in the solid torus. We denote the winding number
of $P_{ij}$ by $w_{ij}$.

Choose pairwise disjoint tubular neighbourhoods $N(C_i)$ of the
components of $C$. For each $i$, choose a zero-framed
orientation-preserving diffeomorphism
\[
f_i\colon S^{1}\times D^{2}\longrightarrow N(C_i).
\]
The link
\[
L
=
\bigcup_{i=1}^{\mu}f_i(P_i)
=
\bigcup_{i=1}^{\mu}\bigcup_{j=1}^{m_i}f_i(P_{ij})
\subset S^{3}
\]
is called the satellite of $C$ with patterns $P_1,\ldots,P_\mu$, and
is denoted by
\[
C(P_1,\ldots,P_\mu).
\]
We denote the component $f_i(P_{ij})$ of $L$ by $L_{ij}$.
\end{cons}

Before stating the result, we introduce some additional notation
associated with the patterns. For each $i$, regard $P_i$ as a link in
a standard unknotted solid torus
\[
V_i\cong S^1\times D^2\subset S^3.
\]
\begin{nota}
Let $\rho_i$ denote the core of the complementary solid torus
\[
S^3\setminus \operatorname{int}(V_i).
\]
We define the \emph{augmented pattern link}
\[
P_i'
=
P_i \cup \rho_i
=
P_{i1}\cup\cdots\cup P_{im_i} \cup \rho_i
\subset S^3.
\]
\end{nota}
\begin{nota}
For each $i$, let
\[
\mathbf t_i=(t_{i1},\ldots,t_{im_i}),
\]
where $t_{ij}$ is the variable corresponding to the component $L_{ij}$ of the satellite link.  Define the monomials
\[
T_i
:=
\prod_{j=1}^{m_i}t_{ij}^{w_{ij}}
\]
and
\[
A_i
:=
\prod_{k=1}^{\mu}T_k^{l_{ik}}
=
\prod_{k=1}^{\mu}\prod_{r=1}^{m_k}
t_{kr}^{\,l_{ik}w_{kr}}.
\]
Since $l_{ii}=0$, the monomial $A_i$ does not involve the variables
belonging to $\mathbf t_i$.
\end{nota}

We can now state the general satellite formula for the Alexander polynomial.

\begin{thm}
\label{thm:main}
Let
\[
L=C(P_1,\ldots,P_\mu)
\]
be the satellite link constructed above. Then
\begin{equation}
\label{eq:full-satellite-formula}
\Delta_L(\mathbf t_1,\ldots,\mathbf t_\mu)
\doteq
\Delta_C(T_1,\ldots,T_\mu)
\prod_{i=1}^{\mu}
\Delta_{P_i'}(\mathbf t_i,A_i).
\end{equation}
\end{thm}
\begin{rem}
\label{rem:change-of-variables}
The substitutions in \eqref{eq:full-satellite-formula} are induced by
the homomorphisms on group rings coming from the inclusions of the
relevant exteriors. More precisely, let
\[
\iota_C\colon X_C\hookrightarrow X_L
\qquad\text{and}\qquad
\iota_i\colon X_{P_i'}\hookrightarrow X_L
\]
denote the natural inclusions. These maps induce ring homomorphisms
\[
\phi_C\colon \mathbb Z[H_1(X_C)]\longrightarrow \mathbb Z[H_1(X_L)]
\qquad\text{and}\qquad
\phi_{P_i}\colon \mathbb Z[H_1(X_{P_i'})]\longrightarrow \mathbb Z[H_1(X_L)].
\]
With respect to the chosen meridian bases, these homomorphisms are
given by
\[
\phi_C(s_i)=T_i
\]
and
\[
\phi_{P_i}(u_{ij})=t_{ij},
\qquad
\phi_{P_i}(z_i)=A_i.
\]
Here, $s_i$ corresponds to a meridian of $C_i$, $u_{ij}$ corresponds
to a meridian of $P_{ij}$, and $z_i$ corresponds to a meridian of the
axis $\rho_i$ of the augmented pattern link $P_i'$. These formulas follow
from a straightforward calculation of linking numbers.
\end{rem}

With the notation introduced in Remark~\ref{rem:change-of-variables}, Theorem~\ref{thm:main} can be reformulated as follows.

\begin{thm}[Reformulation of Theorem~\ref{thm:main}]
\label{thm:main-reformulated}
Let $L=C(P_1,\ldots,P_\mu)$. Then
\begin{equation}
\label{eq:full-satellite-formula}
\Delta_L(\mathbf t_1,\ldots,\mathbf t_\mu)
\doteq
\phi_C\bigl(\Delta_C\bigr)
\prod_{i=1}^{\mu}\phi_{P_i}\bigl(\Delta_{P_i'}\bigr).
\end{equation}
\end{thm}

Theorem~\ref{thm:main}, as well as our other results concerning invariants of satellite links, can be obtained by applying the corresponding single-component formulas successively, one component of the companion at a time.

To formulate this reduction, let
\[
O=S^1\times\{0\}\subset S^1\times D^2
\]
denote the core pattern. Since the satellite operation with pattern
$O$ and zero framing leaves a component unchanged, performing the
satellite operation only on $C_i$ is equivalent to taking $P_j=O$ for
every $j\neq i$. After relabelling the components of $C$, we may assume
that the satellite operation is performed on $C_1$.

This reduces the study of the satellite operation to the following single-component formula.

\begin{prop}
\label{prop:single-component-formula}
Let
\[
L=C(P_1,O,\ldots,O).
\]
For $1\leq j\leq m$, let $t_{1j}$ be the variable corresponding to
the component arising from $P_{1j}$, and, for $2\leq i\leq\mu$, let
$t_i$ be the variable corresponding to the unchanged component $C_i$.
Set
\[
\mathbf t_1=(t_{11},\ldots,t_{1m}),
\qquad
T_1=\prod_{j=1}^{m}t_{1j}^{w_{1j}},
\]
and
\[
\mathbf t_C=(t_{2},\ldots,t_{\mu}),
\qquad
A_1=\prod_{i=2}^{\mu}t_i^{l_{1i}}.
\]
Then
\begin{equation}
\label{eq:single-component-satellite-formula}
\Delta_L(\mathbf t_1,\mathbf t_C)
\doteq
\Delta_C(T_1, \mathbf t_C)
\Delta_{P_1'}(\mathbf t_1,A_1).
\end{equation}
\end{prop}

This formula itself has appeared before in the literature. In particular, it appears in the book of Eisenbud and Neumann~\cite{EisenbudNeumann} in the more general setting of splicing. Their argument also follows a Mayer--Vietoris approach; however, they collapse all variables to powers of a single variable. Note that, given any Laurent polynomial
\[
Q(v_1,\ldots,v_n),
\]
one may choose integers $\alpha_1,\ldots,\alpha_n$ so that the substitution
$v_i=v^{\alpha_i}$ is injective on the finite set of monomials appearing in
$Q$. In this sense, the multivariable polynomial can be recovered from
the one-variable specialization $Q(v^{\alpha_1},\ldots,v^{\alpha_n})$.

For the multivariable Alexander polynomial, this corresponds to passing
to the infinite cyclic cover of the link exterior associated to the
homomorphism
\[
\pi_1(S^3 \setminus L)\longrightarrow \mathbb Z,
\qquad
[\gamma]\longmapsto \sum m_i\,\lk(\gamma,L_i).
\]
The argument then proceeds by a Mayer--Vietoris computation on this cover. 

Another approach to this formula is through the gluing formula for Reidemeister torsion. Such an argument was given by Cimasoni~\cite{Cimasoni}. Moreover, he refined the formula for the Conway potential. 

Finally we note that, behaviour of other invariants of links under satellite operation has also been explained. For instance, see \cite{Signature} for the case of mutlivariable signature.

\subsection{Our approach and higher Alexander polynomials}\hfill

In contrast, we tackle the problem by a direct description followed by a Mayer--Vietoris argument on the maximal abelian cover. The main advantage of this approach is that it retains information about the higher Alexander ideals and the higher Alexander polynomials of satellite links (see Definition~\ref{defn:Alexander polynomials}). Its drawback is that the technical analysis, and even the precise statement of the results, depends on additional hypotheses, most notably the vanishing or nonvanishing conditions on the winding and linking numbers.

We use the following notation for the winding and linking vectors:
\[
\mathbf w_1=(w_{11},\ldots,w_{1m_1})\in\mathbb Z^{m_1},
\qquad
\boldsymbol\ell_1=(l_{12},\ldots,l_{1\mu})\in\mathbb Z^{\mu-1}.
\]
According to whether these vectors vanish, the proofs and results splits into four cases:
\begin{itemize}
\item $\mathbf w_1\neq 0$ and $\boldsymbol\ell_1\neq 0$;
\item $\mathbf w_1=0$ and $\boldsymbol\ell_1\neq 0$;
\item $\mathbf w_1\neq 0$ and $\boldsymbol\ell_1=0$;
\item $\mathbf w_1=0$ and $\boldsymbol\ell_1=0$.
\end{itemize}
We are now ready to state the first of the four case-by-case results.
We begin with the case in which both the winding and linking vectors
are nonzero; this is also the case with the cleanest statement.

\begin{thm}
\label{thm:higher-single-component-case1}
Let
\[
L=C(P_1,O,\ldots,O).
\]
Assume that
\[
\mathbf w_1\neq 0
\qquad\text{and}\qquad
\boldsymbol\ell_1\neq 0.
\]
Then, for each $k\geq 0$, the $k$-th Alexander polynomial of $L$ is given by
\[
\Delta_{L,k}
\;\doteq\;
\gcd_{\substack{\\[2pt] i+j=k }}
\!\Bigl(
    \Delta_{C,i}(T_1,\mathbf t_C)
    \;
    \Delta_{P_1',j}(\mathbf t_1,A_1)
\Bigr),
\]
with the convention that terms with negative indices are omitted
\end{thm}

\begin{rem}\label{rem:indexing Alexander polynomials}
It is common for the Alexander polynomials of a link to be indexed
starting from one rather than zero. However, in this paper we use the
same indexing convention for Fitting ideals and Alexander polynomials,
so we begin at zero. 
\end{rem}
\begin{rem}
\label{rem:case1-recovers-proposition}
When $i=0$, the formula in Theorem~\ref{thm:higher-single-component-case1} recovers Proposition~\ref{prop:single-component-formula} for the ordinary Alexander polynomial.
\end{rem}
In the remaining cases, we will not obtain exact formulas for the higher
Alexander polynomials. Instead, we will obtain upper and lower bounds in
the form of divisibility statements. We begin with the fourth case, since
its statement is the simplest to formulate. To streamline the discussion,
we introduce notation for the relevant sublink of the companion and for
the correction terms that will appear in the formulas below.

\begin{nota}\label{not:sublink}
Let $C'$ denote the sublink of $C$
obtained by deleting the component $C_1$.
\end{nota}

\begin{defn}
\label{defn:DC-DP}
For each $i\geq 0$, define
\[
D_{C,i}:=
\begin{cases}
\Delta_{C',i}, & \mu>2,\\[4pt]
(t_2-1)\Delta_{C',0}, & \mu=2,\ i=0,\\[4pt]
\gcd\bigl(\Delta_{C',i-1},\,(t_2-1)\Delta_{C',i}\bigr), & \mu=2,\ i\geq 1,
\end{cases}
\]
\[
D_{P_1,i}:=
\begin{cases}
\Delta_{P_1,i}, & m_1>1,\\[4pt]
(t_{11}-1)\Delta_{P_1,0}, & m_1=1,\ i=0,\\[4pt]
\gcd\bigl(\Delta_{P_1,i-1},\,(t_{11}-1)\Delta_{P_1,i}\bigr), & m_1=1,\ i\geq 1.
\end{cases}
\]
\end{defn}

\begin{thm}
\label{thm:higher-single-component-case4}
Let
\[
L=C(P_1,O,\ldots,O).
\]
Assume that
\[
\mathbf w_1=0
\qquad\text{and}\qquad
\boldsymbol\ell_1=0.
\]
Then, for each $i\geq 0$, the $k$-th Alexander polynomial of $L$ satisfies
\[
\gcd_{\substack{\\[2pt] i+j=k+1}}
\!\Bigl(
    D_{C,i}(\mathbf t_C)
    \;
    D_{P_1,j}(\mathbf t_1)
\Bigr)
\Bigm|
\Delta_{L,k},
\]
and 
\[
\Delta_{L,k}
\Bigm|
\gcd_{\substack{\\[2pt] i+j=k-2}}
\!\Bigl(
    \Delta_{C',i}(\mathbf t_C)
    \;
    \Delta_{P_1,j}(\mathbf t_1)
\Bigr),
\]
with the convention that terms with negative indices are omitted. In particular, if $\mu>2$ and $m_1>1$, then
\[
\gcd_{\substack{\\[2pt] i+j=k}}
\!\Bigl(
    \Delta_{C',i}(\mathbf t_C)
    \;
    \Delta_{P_1,j}(\mathbf t_1)
\Bigr)
\;\Bigm|\;
\Delta_{L,k}
\;\Bigm|\;
\gcd_{\substack{\\[2pt] i+j=k-2}}
\!\Bigl(
    \Delta_{C',i}(\mathbf t_C)
    \;
    \Delta_{P_1,j}(\mathbf t_1)
\Bigr).
\]
\end{thm}

Finally, we cover the other two cases. 

\begin{thm}
\label{thm:higher-single-component-case2}
Let
\[
L=C(P_1,O,\ldots,O).
\]
Assume that
\[
\mathbf w_1=0
\qquad\text{and}\qquad
\boldsymbol\ell_1\neq 0.
\]
Then, for each $k\geq 0$, the $k$-th Alexander polynomial of $L$ satisfies
\[
\gcd_{\substack{\\[2pt] i+j=k+1}}
\!\Bigl(
    D_{C,i}(\mathbf t_C)
    \;
    \Delta_{P_1',j}(\mathbf t_1,A_1)
\Bigr)
\Bigm|
\Delta_{L,k},
\]
\[
\gcd_{\substack{\\[2pt] i+j=k}}
\!\Bigl(
    D_{C,i}(\mathbf t_C)
    \;
    \Delta_{P_1',j}(\mathbf t_1,A_1)
\Bigr)
\Bigm|
(A_1-1) \Delta_{L,k},
\]
and
\[
\Delta_{L,k}
\Bigm|
\gcd_{\substack{\\[2pt] i+j=k}}
\!\Bigl(
    \Delta_{C',i-1}(\mathbf t_C)
    \;\Delta_{P_1',j}(\mathbf t_1,A_1),\,
(A_1-1)\Delta_{C',i}(\mathbf t_C)
    \;\Delta_{P_1',j}(\mathbf t_1,A_1)
\Bigr)
\]
with the convention that terms with negative indices are omitted.
\end{thm}

\begin{thm}
\label{thm:higher-single-component-case3}
Let
\[
L=C(P_1,O,\ldots,O).
\]
Assume that
\[
\mathbf w_1 \neq 0
\qquad\text{and}\qquad
\boldsymbol\ell_1 = 0.
\]
Then, for each $k\geq 0$, the $k$-th Alexander polynomial of $L$ satisfies
\[
\gcd_{\substack{\\[2pt] i+j=k}}
\!\Bigl(
    \Delta_{C,i}(T_1,\mathbf t_C)
    \;
    D_{P_1,j}(\mathbf t_1)
\Bigr)
\Bigm|
\Delta_{L,k},
\]
\[
\gcd_{\substack{\\[2pt] i+j=k+1}}
\!\Bigl(
    \Delta_{C,i}(T_1,\mathbf t_C)
    \;
    D_{P_1,j}(\mathbf t_1)
\Bigr)
\Bigm|
(T_1-1)\Delta_{L,k},
\]
and
\[
\Delta_{L,k}
\Bigm|
\gcd_{\substack{\\[2pt] i+j=k}}
\!\Bigl(
    \Delta_{C,i-1}(T_1,\mathbf t_C)
    \;\Delta_{P_1,j}(\mathbf t_1),\,
(T_1-1)\Delta_{C,i}(\mathbf t_C)
    \;\Delta_{P_1,j}(\mathbf t_1)
\Bigr)
\]
with the convention that terms with negative indices are omitted.
\end{thm}

\begin{rem}
\label{rem:sublinks}
As the case of vanishing winding or linking number shows, the sublinks
\[
C'\subset C
\qquad\text{and}\qquad
P_1\subset P_1'
\]
naturally appear in the formulas. In the case of the ordinary Alexander
polynomial $\Delta_0$, Torres's formula~\cite{Torres} relates these contributions to
the Alexander polynomial of the original link. For higher Alexander
polynomials, however, such a reduction is no longer available. We
discuss this in more detail in Section~\ref{sec:AlexanderModuleBackground}.

On a separate note, although the formulas themselves do not recover
Proposition~\ref{prop:single-component-formula}, a refinement of the
proof does recover it.
\end{rem}

\begin{rem}\label{rem:no-general-higher-formula}
In the same spirit as Proposition~\ref{prop:single-component-formula} and
Theorem~\ref{thm:main}, we can apply
Theorems~\ref{thm:higher-single-component-case1},
\ref{thm:higher-single-component-case2},
\ref{thm:higher-single-component-case3}, and
\ref{thm:higher-single-component-case4}
successively, one companion component at a time, to obtain a formula for the
most general satellite construction. Since the resulting formulas are more
complicated, we do not record them here as a separate theorem or corollary.
\end{rem}

\begin{defn}
The nullity of a link $L$, denoted by $\beta(L)$, is the rank of its Alexander module. Equivalently,
\[
\beta(L)=\min\{\,i\geq 0 \mid \Delta_{L,i}\neq 0\,\}.
\]
\end{defn}

\begin{cor}
\label{cor:rank-single-component}
Let
\[
L=C(P_1,O,\ldots,O).
\]
Assume that
\[
\mathbf w_1\neq 0
\qquad\text{and}\qquad
\boldsymbol\ell_1\neq 0.
\]
Then 
\[
\beta(L)=\beta(C)+\beta(P_1').
\]
\end{cor}

\begin{rem}
In the remaining cases, the corresponding theorems also give bounds on the nullity of the satellite link. However, these bounds are less clean, so we do not record them here.
\end{rem}

\subsection{Organization of the Paper} \hfill

This paper is organized as follows. To start we cover the background material of the paper in Sections \ref{sec:fitting-ideals-orders} and \ref{sec:AlexanderModuleBackground}. Proofs of the main results are all included in Section \ref{sec:proof-main-results}. Theorems \ref{thm:higher-single-component-case1}, 
\ref{thm:higher-single-component-case4} and \ref{thm:higher-single-component-case2} are proven in Subsections \ref{subsec:winding-linking-nonzero}, \ref{subsec:winding-linking-zero} and \ref{subsec:winding-or-linking-zero}, respectively. The proofs in the
later subsections build on the notation, lemmas, and ideas developed in the preceding subsections. Finally, in Section~\ref{sec:Examples} we give some examples of computations of the higher Alexander polynomials of satellite links and verification of our formulas.

\section*{Acknowledgments} 
We are very grateful to David Cimasoni and Stefan Friedl for their valuable feedback on an earlier draft of this paper. We also wish to thank András Juhász and Livio Ferreti for their insightful suggestions and interest in the project. Finally, the second author was supported by the Additional Funding Programme for Mathematical Sciences, delivered by EPSRC (EP/V521917/1) and the
Heilbronn Institute for Mathematical Research.

\section{Fitting Ideals and Orders}
\label{sec:fitting-ideals-orders}

In this section, we develop the algebraic preliminaries used throughout
the paper, focusing in particular on Fitting ideals and orders of modules.
Throughout this section, we assume that $R$ is a UFD and that all
$R$-modules under consideration are finitely presented.

\begin{defn}
\label{defn:presentation-matrix}
Let $\mathfrak M$ be a finitely presented $R$-module, and let
\[
R^m\xrightarrow{\ \phi\ }R^n\longrightarrow \mathfrak M\longrightarrow 0
\]
be a presentation of $\mathfrak M$ with $n$ generators and $m$ relations,
where $m\geq n$. A \emph{presentation matrix} for $\mathfrak M$ is an
$m\times n$ matrix $M$ representing $\phi$ by right multiplication.
\end{defn}
\begin{rem}
\label{rem:adding-zero-relations}
We may always assume that $m\geq n$ by adding zero rows to the
presentation matrix.
\end{rem}

We now recall the definitions of Fitting ideals and orders. Recall that
an $r$-minor of a matrix is the determinant of an $r\times r$ submatrix.

\begin{defn}
\label{defn:fitting-ideals}
Let $\mathfrak M$ be a finitely presented $R$-module with presentation
matrix $M$ as in Definition~\ref{defn:presentation-matrix}. For $k\geq 0$,
the \emph{$k$-th Fitting ideal of $\mathfrak M$} is the ideal generated by
the $(n-k)$-minors of $M$, that is,
\[
E_k(\mathfrak M)
:=
\left\langle
\text{$(n-k)$-minors of $M$}
\right\rangle
\subseteq R,
\]
with $E_k(\mathfrak M)=R$ for $k\geq n$. The \emph{$k$-th order of
$\mathfrak M$} is defined to be the greatest common divisor of the
elements of $E_k(\mathfrak M)$, that is,
\[
\Delta_k(\mathfrak M):=\gcd E_k(\mathfrak M)\in R/R^\times.
\]
The zeroth order $\Delta_0(\mathfrak M)$ is also called the
\emph{order of $\mathfrak M$} and is denoted by
$\operatorname{ord}(\mathfrak M)$. Note that, each order is defined up to
multiplication by units in $R$.
\end{defn}

We recall the following standard properties of Fitting ideals, which will
be used throughout the paper. For proofs, see, for example,
\cite[Section~15.8]{stacks-project}.

\begin{lem}
\label{lem:basic-fitting-properties}
Let $\mathfrak M$ be a finitely presented $R$-module. Then:
\begin{enumerate}
    \item For every $k\geq 0$, the ideal $E_k(\mathfrak M)$ is
    independent of the choice of presentation of $\mathfrak M$.

    \item The Fitting ideals form an increasing sequence
    \[
    E_0(\mathfrak M)
    \subseteq
    E_1(\mathfrak M)
    \subseteq
    E_2(\mathfrak M)
    \subseteq\cdots.
    \]
    \item The orders form a decreasing sequence with respect to
    divisibility; that is,
    \[
    \Delta_{k+1}(\mathfrak M)\mid \Delta_k(\mathfrak M)
    \qquad\text{for every }k\geq 0.
    \]
\end{enumerate}
\end{lem}

We next recall the behavior of Fitting ideals under a change of base
ring, with particular applications to changes of variables over Laurent
polynomial rings.

\begin{lem}
\label{lem:fitting-base-change}
Let $\psi\colon R\to R'$ be a ring homomorphism, and let $\mathfrak M$
be a finitely presented $R$-module. Then
$\mathfrak M\otimes_R R'$ is a finitely presented $R'$-module, and for
every $k\geq 0$,
\[
E_k(\mathfrak M\otimes_R R')
=
\bigl\langle \psi(E_k(\mathfrak M))\bigr\rangle
\subseteq R'.
\]
In other words, the Fitting ideals of
$\mathfrak M\otimes_R R'$ are generated by the images of the
corresponding Fitting ideals of $\mathfrak M$.
\end{lem}

We next recall the behavior of Fitting ideals in short exact sequences,
which will be used repeatedly in the proofs of our results.

\begin{lem}
\label{lem:Fitting-short-exact-sequence}
Let
\[
0\longrightarrow \mathfrak M'
\longrightarrow \mathfrak M
\longrightarrow \mathfrak M''
\longrightarrow 0
\]
be a short exact sequence of $R$-modules. Then, for $i,j\geq 0$,
\[
E_i(\mathfrak M') \; E_j(\mathfrak M'')
\subseteq
E_{i+j}(\mathfrak M).
\]
Moreover, if the sequence splits, then for every $k\geq 0$,
\[
E_k(\mathfrak M)
=
\sum_{i+j=k}
E_i(\mathfrak M') \; E_j(\mathfrak M'').
\]
\end{lem}

Taking greatest common divisors in Lemma~\ref{lem:Fitting-short-exact-sequence}
gives the following corresponding statements for orders.

\begin{cor}
\label{cor:orders-short-exact-sequence}
Let
\[
0\longrightarrow \mathfrak M'
\longrightarrow \mathfrak M
\longrightarrow \mathfrak M''
\longrightarrow 0
\]
be a short exact sequence of $R$-modules. Then, for every $k\geq 0$,
\[
\Delta_k(\mathfrak M)
\Bigm|
\gcd_{\substack{\\[2pt] i+j=k}}
\!\Bigl(
    \Delta_i(\mathfrak M')
    \;\;
    \Delta_j(\mathfrak M'')
\Bigr).
\]
If the sequence splits, then
\[
\Delta_k(\mathfrak M)
\doteq
\gcd_{\substack{\\[2pt] i+j=k}}
\!\Bigl(
    \Delta_i(\mathfrak M')
    \;\;
    \Delta_j(\mathfrak M'')
\Bigr).
\]
\end{cor}

We next recall the relationship between the zeroth Fitting ideal and the
annihilator of a module. We begin by recalling the definition of the annihilator.

\begin{defn}
\label{defn:annihilator}
Let $\mathfrak M$ be an $R$-module. The \emph{annihilator of $\mathfrak M$}
is the ideal
\[
\operatorname{Ann}(\mathfrak M)
:=
\{\,r\in R \mid r m=0 \text{ for every } m\in\mathfrak M\,\}.
\]
\end{defn}
The annihilator and the zeroth Fitting ideal are related by the following
standard inclusions.

\begin{lem}
\label{lem:annihilator-fitting}
If $\mathfrak M$ is generated by $n$ elements, then
\[
\operatorname{Ann}(\mathfrak M)^n
\subseteq
E_0(\mathfrak M)
\subseteq
\operatorname{Ann}(\mathfrak M).
\]
\end{lem}

We conclude this section by recalling the relationship between the rank
of a module and its orders. We first recall the definition of rank.

\begin{defn}
\label{defn:module-rank}
Let $\mathfrak M$ be an $R$-module, and let $Q(R)$ denote the field of
fractions of $R$. The \emph{rank of $\mathfrak M$} is defined by
\[
\operatorname{rank}_R(\mathfrak M)
:=
\dim_{Q(R)}
\bigl(
Q(R)\otimes_R\mathfrak M
\bigr).
\]
\end{defn}

The following result of Blanchfield describes the orders of a module in
terms of its rank and torsion submodule.

\begin{prop}[\cite{Blanchfield}]
\label{prop:rank-torsion-orders}
Let $\mathfrak M$ be a finitely presented $R$-module, and set
\[
r=\operatorname{rank}_R(\mathfrak M).
\]
Then, for every $i\geq 0$,
\[
\Delta_i(\mathfrak M)
\doteq
\begin{cases}
0, & i<r,\\[4pt]
\Delta_{i-r}\bigl(\operatorname{Tor}(\mathfrak M)\bigr),
& i\geq r.
\end{cases}
\]
Consequently,
\[
\operatorname{rank}_R(\mathfrak M)
=
\min\{\,i\geq 0\mid \Delta_i(\mathfrak M)\neq 0\,\},
\]
and, in particular,
\[
\operatorname{rank}_R(\mathfrak M)=0
\quad\Longleftrightarrow\quad
\Delta_0(\mathfrak M)\neq 0.
\]
\end{prop}
\section{Alexander Module and Maximal Abelian Cover}
\label{sec:AlexanderModuleBackground}

In this section we recover some basic facts about the Alexander modules and the maximal Ableian covers of links. Throughout this section, we use the notation
\[
\Lambda_n
:=
\mathbb Z[t_1^{\pm1},\ldots,t_n^{\pm1}]
\]
for the Laurent polynomial ring in $n$ variables over $\mathbb Z$.

\begin{defn}
\label{defn:max-abelian-cover}
For any oriented link $L\subset S^3$, let $X_L$ denote its exterior. Choose a basepoint
$x_0\in X_L$, and let
\[
p_L\colon \widehat X_L\longrightarrow X_L
\]
be the normal cover corresponding to the kernel of the abelianisation map
\[
\operatorname{ab}_L\colon
\pi_1(X_L,x_0)\longrightarrow H_1(X_L;\mathbb Z).
\]
The cover $\widehat X_L$ is called the \emph{maximal abelian cover} of $L$.
\end{defn}

The deck transformation group of the cover $p_L$ is
\[
H_1(X_L;\mathbb Z)\cong \mathbb Z^\mu,
\]
where $\mu$ is the number of components of $L$. Thus, its group ring is
\[
\mathbb Z[H_1(X_L;\mathbb Z)]\cong \Lambda_\mu.
\]
The action of the deck transformation group gives each homology group
$H_i(\widehat X_L)$ a $\Lambda_\mu$-module structure. Since
$\widehat X_L$ is a non-compact connected $3$-manifold, the homology groups of
primary interest are $H_1(\widehat X_L)$ and $H_2(\widehat X_L)$. We focus on $H_1(\widehat X_L)$. As we will see later in this
section, $H_2(\widehat X_L)$ vanishes in many of the cases we
consider, which further justifies this focus.

\begin{defn}
\label{defn:Alexander polynomials}
The $\Lambda_\mu$-module $H_1(\widehat X_L)$ is called the \emph{Alexander
module} of the oriented link $L$. The $i$-th Fitting ideal of
$H_1(\widehat X_L)$ is called the $i$-th \emph{Alexander ideal} of
$L$ and is denoted by $E_i(L)$. The $i$-th order of
$H_1(\widehat X_L)$ is called the $i$-th \emph{Alexander polynomial}
of $L$ and is denoted by $\Delta_{L,i}$.
\end{defn}

As mentioned in Remark~\ref{rem:indexing Alexander polynomials}, the
zeroth Alexander polynomial in Definition~\ref{defn:Alexander
polynomials} is the classical Alexander polynomial of the link.

\subsection{Fundamental Theorem of Fox Calculus}
\label{subsec:Foxcalc}\hfill

To compute the Alexander polynomials of the link, we need a presentation of the Alexander module. We start by considering a CW complex structure on $X_L$

\begin{prop}
\label{prop:cw-model-link-exterior}
The link exterior $X_L$ is homotopy equivalent to a finite connected
CW complex $\mathcal X$ with one $0$-cell, $n$ $1$-cells, and $n-1$ $2$-cells,
for some $n\geq 1$.
\end{prop}
\begin{proof}
The link exterior $X_L$ is a compact connected $3$-manifold with
nonempty boundary. By the standard relative Morse-theoretic
construction of a handle decomposition, $X_L$ admits a finite handle decomposition with one $0$-handle and no $3$-handles. Let $h_1$ and $h_2$ denote the numbers of $1$- and $2$-handles, respectively. As a result the link exterior $X_L$ is homotopy equivalent to a finite CW complex $\mathcal X_L$ with one $0$-cell, $h_1$ $1$-cells, and $h_2$ $2$-cells.

Since $\chi(X_L)=\frac{\chi(\partial X_L)}{2}=0$, we have
\[
0=\chi(X_L)=1-h_1+h_2,
\]
so $h_2=h_1-1$. Thus $\mathcal X_L$ has the required form.
\end{proof}

We now apply the \emph{Fundamental Theorem of Fox Calculus} to the cell structure on $\mathcal X_L$. In general, this theorem relates the presentation of the fundamental group coming from a CW complex to a presentation of the homology of a normal cover as a module over the group ring of its deck transformation group. We recall the general statement below.

Let $p\colon \widetilde{\mathcal X}\to \mathcal X$ be a regular covering
of a connected $2$-complex. We may assume that $\mathcal X$ has a single
$0$-cell, denoted by $x_0$. The CW structure on $\mathcal X$ determines
a presentation
\[
G=\pi_1(\mathcal X,x_0)
=
\langle s_1,\ldots,s_n\mid r_1,\ldots,r_m\rangle,
\]
where the generators $s_i$ correspond to the $1$-cells and the relators
$r_j$ correspond to the $2$-cells. Let $D$ denote the deck
transformation group of the covering. The quotient homomorphism
$G\to D$ induces a homomorphism of group rings
\[
\varphi\colon \mathbb Z[G]\longrightarrow \mathbb Z[D].
\]
Choose a lift $\widetilde x_0$ of $x_0$ as a basepoint for
$\widetilde{\mathcal X}$, and let
\[
\widetilde{\mathcal X}_0=p^{-1}(x_0).
\]
For each $1$-cell corresponding to $s_i$, choose a lift
$\widetilde s_i$ in $\widetilde{\mathcal X}$ beginning at
$\widetilde x_0$.

\begin{thm}[Fundamental Theorem of Fox Calculus]
\label{thm:Foxcalclus}
The relative homology group
$H_1(\widetilde{\mathcal X},\widetilde{\mathcal X}_0)$ is generated as
a $\mathbb Z[D]$-module by
$\widetilde{s}_1,\ldots,\widetilde{s}_n$, subject to the relations
\[
\sum_{i=1}^n
\varphi\left(\frac{\partial r_j}{\partial s_i}\right)
\widetilde{s}_i=0,
\qquad j=1,\ldots,m.
\]
\end{thm}

At first sight, it is not clear how Theorem~\ref{thm:Foxcalclus}
can be used to compute the Alexander ideals and polynomials, since the
theorem is formulated in terms of relative homology. The following
proposition makes the connection explicit. We use the notation introduced
above for the basepoints; in particular, we take $x_0$ to be the basepoint
of $X_L$ and set
\[
\widehat X_0:=p_L^{-1}(x_0).
\]

\begin{prop}
\label{prop:relative-Alexander-module}
Let
\[
\varepsilon_L\colon \Lambda_\mu\longrightarrow\mathbb Z
\]
be the augmentation homomorphism, and let
\[
\mathfrak a_L
:=
\ker(\varepsilon_L)
=
\left\langle
t_1-1,\ldots,t_\mu-1
\right\rangle
\]
be the augmentation ideal of $\Lambda_\mu$. Then there is a short exact
sequence
\[
0
\longrightarrow
H_1(\widehat X_L)
\longrightarrow
H_1(\widehat X_L,\widehat X_0)
\longrightarrow
\mathfrak a_L
\longrightarrow
0.
\]
Consequently,
\[
\operatorname{Tor} H_1(\widehat X_L,\widehat X_0)
\cong
\operatorname{Tor} H_1(\widehat X_L),
\]
and
\[
\operatorname{rank}_{\Lambda_\mu}
H_1(\widehat X_L,\widehat X_0)
=
\operatorname{rank}_{\Lambda_\mu}
H_1(\widehat X_L)+1.
\]
In particular,
\[
\Delta_{i+1}\bigl(H_1(\widehat X_L,\widehat X_0)\bigr)
\doteq
\Delta_i\bigl(H_1(\widehat X_L)\bigr)
\doteq
\Delta_{L,i}.
\]
\end{prop}

\begin{proof}
The long exact sequence of the pair
$(\widehat X_L,\widehat X_0)$ contains
\[
0
\longrightarrow
H_1(\widehat X_L)
\longrightarrow
H_1(\widehat X_L,\widehat X_0)
\longrightarrow
H_0(\widehat X_0)
\longrightarrow
H_0(\widehat X_L)
\longrightarrow
0.
\]
Choosing a lift of $x_0$ identifies
\[
H_0(\widehat X_0)\cong \Lambda_\mu.
\]
Since $\widehat X_L$ is connected, every deck transformation acts
trivially on $H_0(\widehat X_L)$, and hence
\[
H_0(\widehat X_L)
\cong
\Lambda_\mu/\mathfrak a_L
\cong
\mathbb Z.
\]
Under these identifications, the map
\[
H_0(\widehat X_0)\longrightarrow H_0(\widehat X_L)
\]
is precisely the augmentation homomorphism $\varepsilon_L$. Its kernel
is therefore $\mathfrak a_L$, and exactness gives
\[
0
\longrightarrow
H_1(\widehat X_L)
\longrightarrow
H_1(\widehat X_L,\widehat X_0)
\longrightarrow
\mathfrak a_L
\longrightarrow
0.
\]
Since $\mathfrak a_L$ is a submodule of the free $\Lambda_\mu$-module
$\Lambda_\mu$, it is torsion-free. Moreover, the short exact sequence
\[
0
\longrightarrow
\mathfrak a_L
\longrightarrow
\Lambda_\mu
\xrightarrow{\;\varepsilon_L\;}
\mathbb Z
\longrightarrow
0
\]
shows that
\[
\operatorname{rank}_{\Lambda_\mu}(\mathfrak a_L)=\operatorname{rank}_{\Lambda_\mu}(\Lambda_\mu)-\operatorname{rank}_{\Lambda_\mu}(\mathbb Z)=1,
\]
Now let
\[
z\in
\operatorname{Tor}H_1(\widehat X_L,\widehat X_0).
\]
Its image in $\mathfrak a_L$ is torsion and hence vanishes, since
$\mathfrak a_L$ is torsion-free. By exactness, $z$ therefore lies in
the image of $H_1(\widehat X_L)$. Conversely, every torsion element of
$H_1(\widehat X_L)$ remains torsion in
$H_1(\widehat X_L,\widehat X_0)$. Thus
\[
\operatorname{Tor}H_1(\widehat X_L,\widehat X_0)
\cong
\operatorname{Tor}H_1(\widehat X_L).
\]
Additivity of rank in the short exact sequence gives
\[
\operatorname{rank}_{\Lambda_\mu}
H_1(\widehat X_L,\widehat X_0)
=
\operatorname{rank}_{\Lambda_\mu}
H_1(\widehat X_L)+1.
\]
Therefore, Proposition~\ref{prop:rank-torsion-orders}, together with the
equality of the torsion submodules above, gives
\[
\Delta_{i+1}\bigl(H_1(\widehat X_L,\widehat X_0)\bigr)
\doteq
\Delta_i\bigl(H_1(\widehat X_L)\bigr).
\]
Combined with Definition~\ref{defn:Alexander polynomials}, this proves the final assertion of the proposition.
\end{proof}
Proposition~\ref{prop:relative-Alexander-module} shows that, in order
to compute the Alexander polynomials of $L$, it suffices to obtain a
presentation of the relative homology module
\[
H_1(\widehat X_L,\widehat X_0).
\]
By the Fundamental Theorem of Fox Calculus, such a presentation can
be obtained directly from a suitable presentation of the link group
$\pi_1(X_L)$, such as a Wirtinger or Dehn presentation. This provides
an effective method for computing the Alexander polynomials.

\subsection{Other Abelian Covers and Sublinks}
\label{subsec:other-abelian-covers}\hfill

We next consider abelian covers of link exteriors other than the maximal
abelian cover, with particular emphasis on those associated to sublinks. We begin with the following general divisibility statement.

\begin{lem}
\label{lem:specialization-abelian-cover}
Assume that $\mu\geq 2$. Let
\[
p\colon Z\longrightarrow X_L
\]
be a connected regular cover with free abelian deck transformation group
of rank $n$. Suppose that the cover is determined by an epimorphism
\[
\pi_1(X_L,x_0)
\xrightarrow{\ \operatorname{ab}_L\ }
H_1(X_L;\mathbb Z)\cong\mathbb Z^\mu
\xrightarrow{\ f\ }
\mathbb Z^n.
\]
Let
\[
f_*\colon \Lambda_\mu\longrightarrow\Lambda_n
\]
be the induced homomorphism of group rings, and set
\[
Z_0:=p^{-1}(x_0).
\]
Then, for every $i\geq 0$,
\[
f_*(\Delta_{L,i})
\mid
\Delta_{i+1}\bigl(H_1(Z,Z_0)\bigr).
\]
\end{lem}
\begin{proof}
By the Fundamental Theorem of Fox Calculus, a presentation matrix for
the $\Lambda_n$-module $H_1(Z,Z_0)$ is obtained from a presentation matrix for
$H_1(\widehat X_L,\widehat X_0)$ by applying $f_*$ to each entry.
Consequently,
\[
E_{i+1}\bigl(H_1(Z,Z_0)\bigr)
=
\Bigl \langle f_*\left(
E_{i+1}\bigl(H_1(\widehat X_L,\widehat X_0)\bigr)
\right) \Bigr \rangle.
\]
By Proposition~\ref{prop:relative-Alexander-module},
\[
\Delta_{i+1}\bigl(H_1(\widehat X_L,\widehat X_0)\bigr)
\doteq
\Delta_{L,i}.
\]
Thus $\Delta_{L,i}$ divides every element of
$E_{i+1}(H_1(\widehat X_L,\widehat X_0))$. Applying $f_*$ shows that
$f_*(\Delta_{L,i})$ divides every element of
$E_{i+1}(H_1(Z,Z_0))$, and
therefore
\[
f_*(\Delta_{L,i})
\mid
\Delta_{i+1}\bigl(H_1(Z,Z_0)\bigr).
\]
\end{proof}

\begin{rem}
\label{rem:relative-Alexander-module-abelian-cover}
The argument of Proposition~\ref{prop:relative-Alexander-module} applies
more generally to any connected regular cover with free abelian deck
transformation group. Consequently,
\[
\Delta_{i+1}\bigl(H_1(Z,Z_0)\bigr)
\doteq
\Delta_i\bigl(H_1(Z)\bigr).
\]
\end{rem}
\begin{rem}
\label{rem:delta0-specialization}
In general, the argument of Lemma~\ref{lem:specialization-abelian-cover} does not give a
direct computation of the higher orders of $H_1(Z)$. However, it does
provide an effective method for computing the zeroth order. Consider
the $n\times n$ presentation matrix for
$H_1(\widehat X_L,\widehat X_0)$ arising from a Wirtinger presentation.
The key fact is that its $(n-1)\times(n-1)$ minors
are of the form
\[
(t_i-1)\Delta_L,
\]
and hence, after applying $f_*$, we obtain
\[
\Delta_0\bigl(H_1(Z)\bigr)
=
\gcd_{1\leq i\leq\mu}
\Bigl(
f_*(t_i-1)\,f_*(\Delta_L)
\Bigr).
\]
\end{rem}
We will be particularly interested in the abelian cover obtained by
ignoring one component of the link. Without loss of generality, we
take this to be the first component, leading to the following
definition.
\begin{defn}
\label{defn:special-abelian-cover}
Assume that $\mu\geq 2$. Let
\[
p: Z_L\longrightarrow X_L
\]
denote the connected covering space corresponding to the kernel of the composition
\[
\pi_1(X_L,x_0)
\xrightarrow{\ \mathrm{ab}_L\ }
H_1(X_L;\mathbb Z)
\cong \mathbb Z^\mu
\xrightarrow{\ h\ }
\mathbb Z^{\mu-1},
\]
where $h$ is the projection onto the last $\mu-1$ coordinates.
\end{defn}
Note that, in this case, the induced homomorphism on group rings is the
specialization
\[
f_*\colon
\mathbb Z[t_1^{\pm1},\ldots,t_\mu^{\pm1}]
\longrightarrow
\mathbb Z[t_2^{\pm1},\ldots,t_\mu^{\pm1}],
\qquad
t_1\longmapsto 1.
\]
In particular, Remark~\ref{rem:delta0-specialization} gives
\begin{equation}
\label{eq:zeroth-order-special-cover}
\Delta_0\bigl(H_1(Z_L)\bigr)
\doteq
\begin{cases}
\Delta_L(1,t_2,\ldots,t_\mu), & \mu>2,\\[4pt]
(t_2-1)\,\Delta_L(1,t_2), & \mu=2.
\end{cases}
\end{equation}

A result of Torres~\cite{Torres} now relates the zeroth order of
$H_1(Z_L)$ to the Alexander polynomial of a sublink of $L$. In accordance with Notation~\ref{not:sublink}, let
\[
L':=L\setminus L_1,
\]
and, as in Proposition~\ref{prop:single-component-formula}, set
\[
A_1:=\prod_{j=2}^{\mu} t_j^{\lk(L_1,L_j)}.
\]
\begin{thm}[Torres formula~\cite{Torres}]
\label{thm:Torres}
Let $L=L_1\cup\cdots\cup L_\mu$ be a link with $\mu\geq 2$. Then
\[
\Delta_L(1,t_2,\ldots,t_\mu)
\doteq
\begin{cases}
(A_1-1)\Delta_{L'}(t_2,\ldots,t_\mu), & \mu>2,\\[6pt]
\dfrac{A_1-1}{t_2-1}\,
\Delta_{L'}(t_2), & \mu=2.
\end{cases}
\]
\end{thm}
Combining Theorem~\ref{thm:Torres} with
Equation~\ref{eq:zeroth-order-special-cover}, we obtain
\[
\Delta_0\bigl(H_1(Z_L)\bigr)
\doteq
(A_1-1)\,\Delta_{L'}(t_2,\ldots,t_\mu).
\]

The higher orders of $H_1(Z_L)$ are also related to the higher
Alexander polynomials of the sublink $L'$, although the relationship is
more subtle. The first result in this direction is due to
Traldi~\cite{Traldi}.

Similar to Proposition~\ref{prop:relative-Alexander-module}, let
$\mathfrak a_{L'}$ denote the augmentation ideal of the ring
\[
\Lambda_{\mu-1}
=
\mathbb Z[t_2^{\pm1},\ldots,t_\mu^{\pm1}].
\]
Then Traldi's result can be stated as follows.

\begin{thm}[\cite{Traldi}]
\label{thm:Traldi}
For any link $L$ and any $i\geq 1$, we have
\[
E_{i-1}(L')+(A_1-1)\,E_i(L')
\leq
f_*\bigl(E_i(L)\bigr),
\]
and
\[
f_*\bigl(E_i(L)\bigr)
\leq
E_{i-1}(L')+\mathfrak a_{L'}\,E_i(L').
\]
\end{thm}
Note that, in our setting, as observed in the proof of
Lemma~\ref{lem:specialization-abelian-cover},
\[
f_*\bigl(E_i(L)\bigr)=E_i\bigl(H_1(Z_L)\bigr).
\]
Combining this identification with Theorem~\ref{thm:Traldi} and taking
greatest common divisors of the resulting inequalities, we obtain the
following corollary.
\begin{cor}
\label{cor:traldi-quotient-bounds}
For any link $L$ and any $i\geq 1$,
\[
\Delta_i\bigl(H_1(Z_L)\bigr)
\mid
\gcd\Bigl(
\Delta_{L',i-1},
(A_1-1)\,\Delta_{L',i}
\Bigr).
\]
Furthermore, if $\mu>2$, then
\[
\Delta_{L',i}
\mid
\Delta_i\bigl(H_1(Z_L)\bigr),
\]
whereas if $\mu=2$, then
\[
\gcd\Bigl(
\Delta_{L',i-1},
(t_2-1)\,\Delta_{L',i}
\Bigr)
\mid
\Delta_i\bigl(H_1(Z_L)\bigr).
\]
\end{cor}
\begin{proof}
For any two ideals $J_1$ and $J_2$, we have
\[
\gcd\{\,r\mid r\in J_1+J_2\,\}
=
\gcd\Bigl(
\gcd\{\,r\mid r\in J_1\,\},
\gcd\{\,r\mid r\in J_2\,\}
\Bigr).
\]
Applying this observation to the first inclusion in
Theorem~\ref{thm:Traldi} gives the first divisibility. For the remaining
divisibilities, we apply the same observation to the second inclusion
and use
\[
\gcd\{\,r\mid r\in\mathfrak a_{L'}\,\}
\doteq
\begin{cases}
1, & \mu>2,\\[4pt]
t_2-1, & \mu=2.
\end{cases}
\]
\end{proof}

We record a final relation between the abelian cover $Z_L$ and the sublink $L'$. In particular, we explain how
$H_1(\widehat X_{L'})$ can be recovered as a quotient of $H_1(Z_L)$.

\begin{nota}
\label{not:quotient-by-torus-image}
Let
\[
\Sigma:=\partial N(L_1)
\qquad\text{and}\qquad
\widehat\Sigma:=p^{-1}(\Sigma)\subset Z_L.
\]
Let $\gamma_m\subset\Sigma$ denote a meridian of $L_1$, and set
\[
\widehat\gamma_m:=p^{-1}(\gamma_m)\subset\widehat\Sigma.
\]
Thus, $\widehat\gamma_m$ is the disjoint union of the lifts of
$\gamma_m$ to $\widehat\Sigma$. Let
\[
f_\Sigma\colon H_1(\widehat\Sigma)\longrightarrow H_1(Z_L)
\qquad\text{and}\qquad
f_m\colon H_1(\widehat\gamma_m)\longrightarrow H_1(Z_L)
\]
denote the homomorphisms induced by the inclusions.
\end{nota}

\begin{lem}
\label{lem:quotient-by-torus-image}
There is a natural isomorphism
\[
H_1(Z_L)/\operatorname{im}(f_m)
\cong
H_1(\widehat X_{L'}).
\]
\end{lem}
\begin{proof}
We may assume that the basepoint $x_0$ lies on $\Sigma$. Recall that
\[
\Lambda_{\mu-1}
\cong 
\mathbb Z[t_2^{\pm1},\ldots,t_\mu^{\pm1}].
\]
Consider the maximal abelian cover
\[
p_{L'}\colon \widehat X_{L'}\longrightarrow X_{L'}.
\]
Since
\[
X_{L'}=X_L\cup_\Sigma N(L_1),
\]
this decomposition lifts to a corresponding decomposition of
$\widehat X_{L'}$. The restriction of $p_{L'}$ to $X_L$ is precisely
the cover $Z_L$, while its restriction to $\Sigma$ is the cover
$\widehat\Sigma$. Finally, let
\[
\widehat N:=p_{L'}^{-1}(N(L_1)).
\]
Applying Mayer--Vietoris gives
\begin{equation}
\label{eq:mv-quotient-by-torus-image}
\cdots \longrightarrow
H_1(\widehat\Sigma)
\xrightarrow{\,f_1\,}
H_1(Z_L)\oplus H_1(\widehat N)
\longrightarrow
H_1(\widehat X_{L'})
\longrightarrow
H_0(\widehat\Sigma)
\longrightarrow
H_0(Z_L)\oplus H_0(\widehat N).
\end{equation}
The projection of $f_1$ onto $H_1(Z_L)$ is precisely $f_\Sigma$.
Thus, for some homomorphism
\[
g_1\colon H_1(\widehat\Sigma)\longrightarrow H_1(\widehat N),
\]
we have
\[
f_1=(f_\Sigma,g_1).
\]
We now distinguish two cases according to the covering
$\widehat\Sigma\to\Sigma$. Suppose first that
\[
\operatorname{lk}(L_1,L_j)\neq 0
\]
for some $j$. Then $\widehat\Sigma$ is a disjoint union of infinite
cylinders. In particular,
\[
f_\Sigma\bigl(H_1(\widehat\Sigma)\bigr)
\cong
f_m\bigl(H_1(\widehat\gamma_m)\bigr).
\]
Moreover, $\widehat N$ is a disjoint union of solid cylinders, and hence
\[
H_1(\widehat N)=0.
\]
Furthermore, the map
\[
H_0(\widehat\Sigma)\longrightarrow H_0(\widehat N)
\]
is injective. Exactness of
\eqref{eq:mv-quotient-by-torus-image} therefore gives
\[
\begin{aligned}
H_1(\widehat X_{L'})
&\cong
\bigl(H_1(Z_L)\oplus H_1(\widehat N)\bigr)/
\operatorname{im}(f_1)\\
&\cong
H_1(Z_L)/\operatorname{im}(f_\Sigma)\\
&\cong
H_1(Z_L)/\operatorname{im}(f_m).
\end{aligned}
\]
Now suppose that
\[
\operatorname{lk}(L_1,L_j)=0
\qquad\text{for all }j.
\]
In this case, $\widehat\Sigma$ is a disjoint union of tori and
$\widehat N$ is a disjoint union of solid tori. As
$\Lambda_{\mu-1}$-modules, the relevant part of the Mayer--Vietoris map
takes the form
\[
\Lambda_{\mu-1}^2
\xrightarrow{\,f_1\,}
H_1(Z_L)\oplus\Lambda_{\mu-1},
\]
where
\[
f_1=
\begin{bmatrix}
f_m & 0\\
* & \operatorname{id}
\end{bmatrix}.
\]
It follows that
\[
\begin{aligned}
H_1(\widehat X_{L'})
&\cong
\bigl(H_1(Z_L)\oplus H_1(\widehat N)\bigr)/
\operatorname{im}(f_1)\\
&\cong
H_1(Z_L)/\operatorname{im}(f_m).
\end{aligned}
\]
This proves the result.
\end{proof}

\subsection{Second Homology of the Maximal Abelian Cover}
\label{subsec:second-homology-max-abelian-cover}\hfill 

We conclude this section by studying the second homology group
$H_2(\widehat X_L)$ of the maximal abelian cover. We refer the reader to Chapter 339, especially Proposition 339.19, in Freidl~\cite{Friedl} for an alternative approach to this topic. 

Our main goal is to
establish the following vanishing result. 

\begin{prop}
\label{prop:vanishingH2}
Let $L$ be a link in $S^3$. Then
\[
\operatorname{rank}_{\Lambda_{\mu}}
\bigl(H_2(\widehat X_L)\bigr)
=
\operatorname{rank}_{\Lambda_{\mu}}
\bigl(H_1(\widehat X_L)\bigr).
\]
In particular, if the Alexander polynomial $\Delta_L$ is nonzero, then
\[
H_2(\widehat X_L)=0.
\]
\end{prop}

We start from Proposition~\ref{prop:cw-model-link-exterior} and lift the CW structure to $\widehat X_L$. This gives the following cellular
description.

\begin{cor}
\label{cor:cellular-chain-complex-max-abelian-cover}
The homology groups $H_*(\widehat X_L)$ are computed by the cellular
chain complex
\[
0
\longrightarrow
C_2
\xrightarrow{\partial_2}
C_1
\xrightarrow{\partial_1}
C_0
\longrightarrow
0,
\]
where each $C_i$ is the free $\Lambda_\mu$-module generated by the lifts
of the $i$-cells of $\mathcal X$. In particular,
\[
\operatorname{rank}_{\Lambda_\mu} C_1=n
\qquad\text{and}\qquad
\operatorname{rank}_{\Lambda_\mu} C_2=n-1.
\]
\end{cor}
Since
\[
H_2(\widehat X_L)=\ker(\partial_2),
\]
the module $H_2(\widehat X_L)$ is a submodule of the free module $C_2$.
We will use the following elementary lemma.
\begin{lem}
\label{lem:rank-of-submodules-of-free-modules}
Let $M$ be a free $R$-module and let $N\subseteq M$ be a submodule.
If
\[
\operatorname{rank}_R(N)=0,
\]
then $N=0$.
\end{lem}
\begin{proof}
Since $N$ is a submodule of a free $R$-module, it is torsion-free.
Therefore, the natural map
\[
N\longrightarrow Q(R)\otimes_R N
\]
is injective. On the other hand, the assumption
$\operatorname{rank}_R(N)=0$ implies
\[
Q(R)\otimes_R N=0.
\]
Hence $N=0$.
\end{proof}

Now we are ready to prove Proposition~\ref{prop:vanishingH2}.

\begin{proof}
Consider the exact sequence
\[
0\longrightarrow \ker(\partial_2)
\longrightarrow C_2
\longrightarrow \operatorname{im}(\partial_2)
\longrightarrow 0.
\]
By additivity of rank,
\[
\operatorname{rank}_{\Lambda_\mu}(C_2)
=
\operatorname{rank}_{\Lambda_\mu}(\ker(\partial_2))
+
\operatorname{rank}_{\Lambda_\mu}(\operatorname{im}(\partial_2)).
\]
Similarly,
\[
\operatorname{rank}_{\Lambda_\mu}(\ker(\partial_1))
=
\operatorname{rank}_{\Lambda_\mu}(C_1)
-
\operatorname{rank}_{\Lambda_\mu}(\operatorname{im}(\partial_1)).
\]
Since $\widehat X_L$ is connected,
\[
H_0(\widehat X_L)
\cong
\Lambda_\mu/
\langle t_1-1,\ldots,t_\mu-1\rangle
\cong
\mathbb Z.
\]
Hence
\[
\operatorname{im}(\partial_1)
=
\langle t_1-1,\ldots,t_\mu-1\rangle.
\]
Therefore,
\[
\operatorname{rank}_{\Lambda_\mu}(\operatorname{im}(\partial_1))=1,
\]
and consequently
\[
\operatorname{rank}_{\Lambda_\mu}(\ker(\partial_1))
=
n-1.
\]
Since
\[
H_1(\widehat X_L)
=
\ker(\partial_1)/\operatorname{im}(\partial_2),
\]
we obtain
\[
\operatorname{rank}_{\Lambda_\mu}(H_1(\widehat X_L))
=
n-1-
\operatorname{rank}_{\Lambda_\mu}(\operatorname{im}(\partial_2)).
\]
On the other hand, since
\[
\operatorname{rank}_{\Lambda_\mu}(C_2)=n-1,
\]
we have
\[
\operatorname{rank}_{\Lambda_\mu}(H_2(\widehat X_L))
=
\operatorname{rank}_{\Lambda_\mu}(\ker(\partial_2))
=
n-1-
\operatorname{rank}_{\Lambda_\mu}(\operatorname{im}(\partial_2)).
\]
Thus
\[
\operatorname{rank}_{\Lambda_\mu}(H_2(\widehat X_L))
=
\operatorname{rank}_{\Lambda_\mu}(H_1(\widehat X_L)).
\]
If $\Delta_L\neq 0$, then the Alexander module
$H_1(\widehat X_L)$ is $\Lambda_\mu$-torsion, and hence
\[
\operatorname{rank}_{\Lambda_\mu}(H_1(\widehat X_L))=0.
\]
Therefore,
\[
\operatorname{rank}_{\Lambda_\mu}(H_2(\widehat X_L))=0.
\]
By Lemma~\ref{lem:rank-of-submodules-of-free-modules},
\[
H_2(\widehat X_L)=0.
\]
\end{proof}

\section{Proofs of the Main Results}
\label{sec:proof-main-results}

\subsection{From One Component to the General Satellite}
\label{subsec:one-to-many-components}\hfill

As explained in the introduction, the general satellite construction can be obtained by performing the satellite operation one component at a time. As noted in Remark~\ref{rem:no-general-higher-formula}, our main results focus on the one-component case. Here, however, we briefly describe how the single-component formula, Proposition~\ref{prop:single-component-formula}, extends to the general formula, Theorem~\ref{thm:main}, for the ordinary Alexander polynomial. The reformulation in Theorem~\ref{thm:main-reformulated} and the coefficient maps of Remark~\ref{rem:change-of-variables} allow us to keep track of the variables. 

\begin{proof}[Proof of Theorem~\ref{thm:main}]
For $0\leq j\leq \mu$, set
\[
L^{(j)}:=C(P_1,\ldots,P_j,O,\ldots,O),
\]
so that $L^{(0)}=C$ and $L^{(\mu)}=L$. For each $j\geq 1$, let
\[
\eta_j\colon X_{L^{(j-1)}}\hookrightarrow X_{L^{(j)}}
\]
denote the natural inclusion, and let
\[
\psi_j\colon
\mathbb Z[H_1(X_{L^{(j-1)}})]
\longrightarrow
\mathbb Z[H_1(X_{L^{(j)}})]
\]
be the induced homomorphism of group rings. We also write
\[
\phi_C^{(j)}\colon \mathbb Z[H_1(X_C)]\longrightarrow \mathbb Z[H_1(X_{L^{(j)}})]
\qquad\text{and}\qquad
\phi_{P_i}^{(j)}\colon \mathbb Z[H_1(X_{P_i'})]\longrightarrow \mathbb Z[H_1(X_{L^{(j)}})]
\]
for the coefficient homomorphisms induced by the inclusions of
$X_C$ and $X_{P_i'}$ into $X_{L^{(j)}}$.

We claim that, for each $1\leq j\leq\mu$,
\begin{equation}
\label{eq:inductive-reformulation}
\Delta_{L^{(j)}}
\doteq
\phi_C^{(j)}\bigl(\Delta_C\bigr)
\prod_{i=1}^{j}
\phi_{P_i}^{(j)}\bigl(\Delta_{P_i'}\bigr).
\end{equation}
For $j=1$, this is precisely
Proposition~\ref{prop:single-component-formula}. Assume that
\eqref{eq:inductive-reformulation} holds for some $j<\mu$.

Apply Proposition~\ref{prop:single-component-formula} to the satellite
\[
L^{(j+1)}
=
L^{(j)}
\bigl(
\underbrace{O,\ldots,O}_{j},
P_{j+1},
\underbrace{O,\ldots,O}_{\mu-j-1}
\bigr).
\]
Using the reformulation in Theorem~\ref{thm:main-reformulated}, we obtain
\[
\Delta_{L^{(j+1)}}
\doteq
\psi_{j+1}\bigl(\Delta_{L^{(j)}}\bigr)\,
\phi_{P_{j+1}}^{(j+1)}
\bigl(\Delta_{P_{j+1}'}\bigr).
\]
By functoriality of the induced maps on group rings,
\[
\psi_{j+1}\circ \phi_C^{(j)}=\phi_C^{(j+1)}
\qquad\text{and}\qquad
\psi_{j+1}\circ \phi_{P_i}^{(j)}=\phi_{P_i}^{(j+1)}
\quad (1\leq i\leq j).
\]
Applying $\psi_{j+1}$ to
\eqref{eq:inductive-reformulation} therefore gives
\[
\psi_{j+1}\bigl(\Delta_{L^{(j)}}\bigr)
\doteq
\phi_C^{(j+1)}\bigl(\Delta_C\bigr)
\prod_{i=1}^{j}
\phi_{P_i}^{(j+1)}\bigl(\Delta_{P_i'}\bigr).
\]
Hence
\[
\Delta_{L^{(j+1)}}
\doteq
\phi_C^{(j+1)}\bigl(\Delta_C\bigr)
\prod_{i=1}^{j+1}
\phi_{P_i}^{(j+1)}\bigl(\Delta_{P_i'}\bigr),
\]
so \eqref{eq:inductive-reformulation} holds for $j+1$.

By induction, \eqref{eq:inductive-reformulation} holds for $j=\mu$.
Since $L^{(\mu)}=L$ and the maps $\phi_C^{(\mu)}$ and
$\phi_{P_i}^{(\mu)}$ are precisely those of
Remark~\ref{rem:change-of-variables}, the resulting formula is
equivalent to Theorem~\ref{thm:main}.
\end{proof}

\subsection{Setup and Conventions}
\label{subsec:setup-and-conventions}\hfill

In this subsection, we establish the geometric setup and notation that will be used throughout the proof of Theorems~\ref{thm:higher-single-component-case1}, \ref{thm:higher-single-component-case4}, \ref{thm:higher-single-component-case2}, and \ref{thm:higher-single-component-case3}. We study the maximal
abelian cover of the satellite exterior and the decomposition induced by the companion and pattern pieces.

\begin{conv}
\label{conv:single-component-setup}
Throughout the following subsections, we use the following conventions:
\begin{itemize}[leftmargin=*]
\item To simplify notation, we henceforth write $P$ in place of $P_1$
for the distinguished pattern link, and $P'$ for the corresponding
augmented pattern. We continue
to distinguish the components of $P$ and $P'$ by the notations
\[
P=P_{11}\cup\cdots\cup P_{1m},
\qquad
P'=P\cup\rho,
\]
where $m$ denotes the number of components of $P$.
\item We consider the satellite link
\[
L=C(P,O,\ldots,O).
\]
\item We write
\[
X_L=S^3\setminus\operatorname{int}N(L),
\qquad
X_C=S^3\setminus\operatorname{int}N(C),
\]
and regard $X_C$ as a subspace of $X_L$.
\item We let
\[
X_{P}
:=
(S^1\times D^2)\setminus\operatorname{int}N(P_1)
\]
denote the exterior of $P$ in the solid torus, identified with its
image in $N(C_1)\subset S^3$.
\item We set
\[
\Sigma:=\partial N(C_1)
\]
and choose a basepoint $x_0\in\Sigma$.
\item We let
\[
p_L\colon \widehat X_L\longrightarrow X_L,
\qquad
\widehat X_C\longrightarrow X_C,
\qquad
\widehat X_{P}\longrightarrow X_{P}
\]
denote the maximal abelian covers of $X_L$, $X_C$, and $X_{P}$,
respectively.
\item We set
\[
\widehat X_0:=p_L^{-1}(x_0)\subseteq \widehat X_L,
\]
the full preimage of the basepoint $x_0$.
\item We identify the deck transformation group of $\widehat X_L$ with
\[
H_1(X_L;\mathbb Z)
\cong
\mathbb Z^{m}\oplus\mathbb Z^{\mu-1},
\]
using the ordered basis corresponding to
\[
(t_{11},\ldots,t_{1m},t_2,\ldots,t_\mu) = (\mathbf t_1,\mathbf t_C).
\]
\item The corresponding group ring is
\[
R:=\mathbb Z\bigl[t_{11}^{\pm1},\ldots,t_{1m}^{\pm1},
t_2^{\pm1},\ldots,t_\mu^{\pm1}\bigr].
\]
\item Using the notation of Remark~\ref{rem:change-of-variables}, we set
\[
R_C
:=
\mathbb Z[H_1(X_C)]
=
\mathbb Z[s_1^{\pm1},\ldots,s_\mu^{\pm1}],
\]
and
\[
R_P
:=
\mathbb Z[H_1(X_{P'})]
=
\mathbb Z[u_{11}^{\pm1},\ldots,u_{1m}^{\pm1},z_1^{\pm1}].
\]
We regard $R$ as an $R_C$-module via $\phi_C$ and as an
$R_P$-module via $\phi_P$.
\item Unless explicitly stated otherwise, all homology groups are regarded as $R$-modules; when a homology group is instead regarded as an $R_C$-module or an $R_P$-module, this will be stated explicitly.
\item We use the notation 
\[
\Delta_{L,i},\qquad
\Delta_{P,i},\qquad
\Delta_{C,i},
\]
and analogous notation for other links, to denote the
$i$-th Alexander polynomial of the corresponding link. We primarily
regard $P$ as a link in the solid torus, but when considering its
Alexander polynomial, we regard the solid torus as embedded as a
standard unknotted solid torus in $S^3$.
\item Throughout the remainder of this section, terms involving
Alexander ideals or Alexander polynomials with negative indices are
understood to be omitted.
\end{itemize}
\end{conv}
With these conventions in place, we now introduce the lifted pieces that will be used throughout the proofs.

\begin{defn}
\label{defn:Y-subspaces}
Define
\[
Y_C:=p_L^{-1}(X_C),
\qquad
Y_P:=p_L^{-1}(X_P),
\qquad
\widehat\Sigma:=p_L^{-1}(\Sigma).
\]
\end{defn}

\begin{rem}\label{rem:lifted}
The spaces $Y_C$, $Y_P$, and $\widehat\Sigma$ need not be connected.
The restrictions of $p_L$ make them covering spaces of $X_C$, $X_P$,
and $\Sigma$, respectively. Moreover, the decomposition
\[
X_L=X_C\cup_\Sigma X_P
\]
lifts to
\[
\widehat X_L=Y_C\cup_{\widehat\Sigma}Y_P.
\]
\end{rem}

We need a few more technical notations before we proceed.

\begin{nota}
We use the following notation for the winding and linking vectors:
\[
\mathbf w_1=(w_{11},\ldots,w_{1m})\in\mathbb Z^{m},
\qquad
\boldsymbol\ell_1=(l_{12},\ldots,l_{1\mu})
\in\mathbb Z^{\mu-1}.
\]
We also consider the corresponding extended vectors
\[
\overline{\mathbf w}_1=(\mathbf w_1,\mathbf 0),
\qquad
\overline{\boldsymbol\ell}_1=(\mathbf 0,\boldsymbol\ell_1)
\in
\mathbb Z^{m}\oplus\mathbb Z^{\mu-1}.
\]
\end{nota}
\begin{defn}
\label{defn:component-indexing-sets}
Define the following indexing sets:
\[
\Gamma_C
:=
\mathbb Z^{m}/\langle\mathbf w_1\rangle,
\qquad
\Gamma_P
:=
\mathbb Z^{\mu-1}/\langle\boldsymbol\ell_1\rangle,
\]
and
\[
\Gamma_\Sigma
:=
\left(\mathbb Z^{m}\oplus\mathbb Z^{\mu-1}\right)
\Big/
\left\langle
\overline{\mathbf w}_1,\overline{\boldsymbol\ell}_1
\right\rangle.
\]
\end{defn}
\begin{nota}
For
\[
\mathbf n=(n_{11},\ldots,n_{1m})\in\mathbb Z^{m},
\qquad
\mathbf r=(r_2,\ldots,r_\mu)\in\mathbb Z^{\mu-1},
\]
we write
\[
\mathbf t_1^{\mathbf n}
:=
\prod_{j=1}^{m}t_{1j}^{n_{1j}},
\qquad
\mathbf t_C^{\mathbf r}
:=
\prod_{i=2}^{\mu}t_i^{r_i}.
\]
\end{nota}

We end this subsection by describing the main tools that drive the
proofs in the following subsections. These are the absolute and
relative Mayer--Vietoris sequences associated to the lifted
decomposition.

\begin{cons}[Mayer--Vietoris sequences]
The decomposition from Remark~\ref{rem:lifted},
\[
\widehat X_L=Y_C\cup_{\widehat\Sigma}Y_P,
\]
gives the following absolute Mayer--Vietoris sequence of $R$-modules:
\begin{equation}
\label{eq:absolute-MV-sequence}
\cdots
\longrightarrow
H_1(\widehat\Sigma)
\longrightarrow
H_1(Y_C)\oplus H_1(Y_P)
\longrightarrow
H_1(\widehat X_L)
\longrightarrow
H_0(\widehat\Sigma)
\longrightarrow
H_0(Y_C)\oplus H_0(Y_P).
\end{equation}
Since
\[
\widehat X_0\subseteq\widehat\Sigma\subseteq Y_C,Y_P,
\]
the same decomposition gives the relative Mayer--Vietoris sequence
\begin{equation}
\label{eq:relative-MV-sequence}
H_2(\widehat X_L,\widehat X_0)
\longrightarrow
H_1(\widehat\Sigma,\widehat X_0)
\longrightarrow
H_1(Y_C,\widehat X_0)
\oplus
H_1(Y_P,\widehat X_0)
\longrightarrow
H_1(\widehat X_L,\widehat X_0)
\longrightarrow
H_0(\widehat\Sigma,\widehat X_0).
\end{equation}
\end{cons}

\subsection{Winding and Linking Number Non-Zero}
\label{subsec:winding-linking-nonzero}\hfill

\noindent
Throughout this subsection, we assume that
\[
\mathbf w_1\neq 0
\qquad\text{and}\qquad
\boldsymbol\ell_1\neq 0.
\]
Our main goal in this subsection is to prove
Theorem~\ref{thm:higher-single-component-case1}.

We first focus on the absolute Mayer--Vietoris sequence
\eqref{eq:absolute-MV-sequence}. We begin by analyzing the
$R$-module structure of the zeroth homology groups of the lifted pieces.

\begin{lem}
\label{lem:lifted-pieces-nonzero}
Let $\widehat X_L$, $Y_C$, $Y_P$, and $\widehat\Sigma$ be as in
Definition~\ref{defn:Y-subspaces}. Then there are identifications
\begin{align}
Y_C
&=
\bigsqcup_{[\mathbf n]\in\Gamma_C}
\mathbf t_1^{\mathbf n}\widehat X_C,
\label{eq:YC-components}
\\
Y_P
&=
\bigsqcup_{[\mathbf r]\in\Gamma_P}
\mathbf t_C^{\mathbf r}\widehat X_P,
\label{eq:YP-components}
\\
\widehat\Sigma
&=
\bigsqcup_{[(\mathbf n,\mathbf r)]\in\Gamma_\Sigma}
\mathbf t_1^{\mathbf n}\mathbf t_C^{\mathbf r}\widetilde\Sigma,
\label{eq:Sigma-components}
\end{align}
where $\widetilde\Sigma\cong\mathbb R^2$ denotes the universal cover of $\Sigma$.
\end{lem}

\begin{proof}
We prove the identification~\eqref{eq:YC-components}; the other two
are proved similarly.

Let $\widehat x_0\in\widehat\Sigma$ be a lift of the basepoint $x_0$, and
let $Y_C^0$ denote the connected component of $Y_C$ containing
$\widehat x_0$. The restriction
\[
p_L|_{Y_C^0}\colon Y_C^0\longrightarrow X_C
\]
is the connected covering corresponding to the kernel of the composition
\[
\pi_1(X_C,x_0)
\xrightarrow{\,\iota_*\,}
\pi_1(X_L,x_0)
\xrightarrow{\,\operatorname{ab}_L\,}
H_1(X_L;\mathbb Z),
\]
where $\iota\colon X_C\hookrightarrow X_L$ is the inclusion.

Indeed, if $\gamma$ is a loop in $X_C$ based at $x_0$, then its lift
to $\widehat X_L$ beginning at $\widehat x_0$ is contained in $Y_C^0$,
and this lift is closed if and only if
\[
\iota_*[\gamma]=0
\qquad\text{in }H_1(X_L;\mathbb Z).
\]
Here $\iota_*$ denotes the homomorphism induced by the inclusion on first
homology,
\[
\iota_*\colon
H_1(X_C;\mathbb Z)\longrightarrow H_1(X_L;\mathbb Z).
\]
With respect to the conventions fixed above, this map is given by
\[
\iota_*\colon
\mathbb Z^\mu
\longrightarrow
\mathbb Z^{m}\oplus\mathbb Z^{\mu-1},
\qquad
(a_1,\ldots,a_\mu)
\longmapsto
(a_1\mathbf w_1,a_2,\ldots,a_\mu).
\]
Since $\mathbf w_1\neq 0$, this map is injective. Consequently,
\[
\iota_*[\gamma]=0
\quad\Longleftrightarrow\quad
[\gamma]=0\text{ in }H_1(X_C;\mathbb Z).
\]
It follows that
\[
\ker\!\left(
\operatorname{ab}_L\circ\iota_*
\right)
=
\ker\!\left(
\operatorname{ab}_C
\right),
\]
and hence $Y_C^0$ is isomorphic, as a covering space of $X_C$, to the
maximal abelian cover $\widehat X_C$. Since $\widehat X_L\to X_L$ is
regular, every connected component of $Y_C$ is a deck translate of
$Y_C^0$, and hence isomorphic to $\widehat X_C$.

It remains to identify $\pi_0(Y_C)$. Write
\[
G_L:=\pi_1(X_L,x_0).
\]
The components of $Y_C$ are naturally indexed by the double coset space
\[
[G_L,G_L]\backslash G_L/\operatorname{im}(\iota_*).
\]
Passing to abelianisation, we obtain
\[
\pi_0(Y_C)
\cong
H_1(X_L;\mathbb Z)/\operatorname{im}(\iota_*)
\cong
\mathbb Z^{m}/\langle\mathbf w_1\rangle
=
\Gamma_C.
\]
Consequently,
\[
Y_C
=
\bigsqcup_{[\mathbf n]\in\Gamma_C}
\mathbf t_1^{\mathbf n}\widehat X_C,
\]
which proves~\eqref{eq:YC-components}.
\end{proof}

\begin{rem}
With the above identifications, if
\[
\mathbf t_1^{\mathbf n}\widehat X_C
\qquad\text{and}\qquad
\mathbf t_C^{\mathbf r}\widehat X_P
\]
are corresponding connected components of $Y_C$ and $Y_P$, then their
intersection is the corresponding translate of the lifted gluing
surface:
\[
\mathbf t_1^{\mathbf n}\widehat X_C
\cap
\mathbf t_C^{\mathbf r}\widehat X_P
=
\mathbf t_1^{\mathbf n}\mathbf t_C^{\mathbf r}\widetilde T.
\]
\end{rem}
In the next step, we compute the $R$-module structure of
$H_1(\widehat\Sigma)$, $H_1(Y_C)$, and $H_1(Y_P)$, and then analyze
the maps in the Mayer--Vietoris sequence~\eqref{eq:absolute-MV-sequence}.

\begin{prop}
\label{prop:homology-of-lifted-pieces}
There are natural isomorphisms of $R$-modules
\begin{align}
H_1(\widehat\Sigma)
&=0,
\label{eq:H1-lifted-torus}
\\
H_1(Y_C)
&\cong
H_1(\widehat X_C)\otimes_{R_C}R,
\label{eq:H1-YC-extension}
\\
H_1(Y_P)
&\cong
H_1(\widehat X_P)\otimes_{R_P}R.
\label{eq:H1-YP-extension}
\end{align}
\end{prop}

\begin{proof}
By Lemma~\ref{lem:lifted-pieces-nonzero}, every connected component of
$\widehat\Sigma$ is isomorphic to the universal cover
$\widetilde\Sigma\cong\mathbb R^2$. Therefore,
\[
H_1(\widehat\Sigma)=0.
\]

The homomorphism on deck transformation groups associated with the
inclusion $X_C\hookrightarrow X_L$ is given by
\[
s_1\longmapsto T_1,
\qquad
s_i\longmapsto t_i
\quad (2\leq i\leq\mu),
\]
and hence induces the coefficient homomorphism $\phi_C$ defined in
Remark~\ref{rem:change-of-variables}. The decomposition
\[
Y_C
=
\bigsqcup_{[\mathbf n]\in\Gamma_C}
\mathbf t_1^{\mathbf n}\widehat X_C
\]
therefore identifies the cellular chain complex of $Y_C$ with the
induced chain complex
\[
C_*(Y_C)
\cong
C_*(\widehat X_C)\otimes_{R_C}R.
\]
Since $\mathbf w_1\neq 0$, the homomorphism $\phi_C$ is injective.
Moreover, $R$ is a free, and hence flat, $R_C$-module. It follows that
extension of scalars commutes with homology, giving
\[
H_1(Y_C)
\cong
H_1(\widehat X_C)\otimes_{R_C}R.
\]
The same argument, applied to the pattern piece, gives
\[
H_1(Y_P)
\cong
H_1(\widehat X_P)\otimes_{R_P}R.
\]
\end{proof}
\begin{rem}
\label{rem:H0-lifted-pieces}
The same tensor-product description holds in degree $0$. Moreover,
since $\widehat X_C$, $\widehat X_P$, and $\widetilde\Sigma$ are
connected, the computation of $H_0$ follows directly from the proof of
Lemma~\ref{lem:lifted-pieces-nonzero}. In particular,
\begin{align}
H_0(Y_C)
&\cong
R\big/
\langle T_1-1,t_2-1,\ldots,t_\mu-1\rangle,
\label{eq:H0-YC}
\\
H_0(Y_P)
&\cong
R\big/
\langle
t_{11}-1,\ldots,t_{1m}-1,A_1-1
\rangle,
\label{eq:H0-YP}
\\
H_0(\widehat\Sigma)
&\cong
R\big/\langle T_1-1,A_1-1\rangle.
\label{eq:H0-lifted-Sigma}
\end{align}
\end{rem}
We now isolate the kernel of the map between the zeroth homology groups
in the absolute Mayer--Vietoris sequence.

\begin{defn}
\label{defn:K-module}
Let
\[
f\colon
H_0(\widehat\Sigma)
\longrightarrow
H_0(Y_C)\oplus H_0(Y_P)
\]
denote the homomorphism appearing in
\eqref{eq:absolute-MV-sequence}. Define
\[
\mathfrak K:=\ker(f).
\]
\end{defn}
Combining Definition~\ref{defn:K-module} with
Equation~\eqref{eq:H1-lifted-torus}, the Mayer--Vietoris sequence
\eqref{eq:absolute-MV-sequence} reduces to the short exact sequence
\begin{equation}
\label{eq:MV-short-exact}
0
\longrightarrow
H_1(Y_C)\oplus H_1(Y_P)
\longrightarrow
H_1(\widehat X_L)
\longrightarrow
\mathfrak K
\longrightarrow
0.
\end{equation}
The remaining task is to understand the module $\mathfrak K$. We first
show that all of its orders are trivial.

\begin{prop}
\label{prop:K-order}
The module $\mathfrak K$ satisfies
\[
\Delta_i(\mathfrak K)=1
\qquad\text{for all } i\geq 0.
\]
\end{prop}
\begin{proof}
By definition, $\mathfrak K$ is an $R$-submodule of
$H_0(\widehat\Sigma)$. Equation~\eqref{eq:H0-lifted-Sigma} shows that
$H_0(\widehat\Sigma)$ is finitely generated over $R$. Since $R$ is
Noetherian, the submodule $\mathfrak K$ is also finitely generated.

If $\mathfrak K=0$, then $E_0(\mathfrak K)=R$, and hence
$\Delta_0(\mathfrak K)=1$. We may therefore assume that
$\mathfrak K$ is generated by $N\geq 1$ elements. Using Lemma~\ref{lem:annihilator-fitting}, we have 
\[
\operatorname{Ann}(\mathfrak K)^N
\subseteq
E_0(\mathfrak K).
\]
Since $\mathfrak K$ is a submodule of
$R/\langle T_1-1,A_1-1\rangle$, both $T_1-1$ and $A_1-1$
annihilate $\mathfrak K$. It follows that
\[
(T_1-1)^N,(A_1-1)^N
\in E_0(\mathfrak K).
\]
The monomial $T_1$ involves only the variables
$t_{11},\ldots,t_{1m}$, whereas $A_1$ involves only the variables
$t_2,\ldots,t_\mu$. Moreover, under the standing assumptions
$\mathbf w_1\neq 0$ and $\boldsymbol\ell_1\neq 0$, neither monomial is
equal to $1$. Consequently,
\[
\gcd\bigl((T_1-1)^N,(A_1-1)^N\bigr)=1.
\]
Since $\Delta_0(\mathfrak K)$ is the greatest common divisor of the
elements of $E_0(\mathfrak K)$, we conclude that
\[
\Delta_0(\mathfrak K)=1.
\]
Finally, by Lemma~\ref{lem:basic-fitting-properties},
\[
\Delta_{i+1}(\mathfrak K)\mid\Delta_i(\mathfrak K)
\qquad\text{for all }i\geq0.
\]
It follows inductively that
\[
\Delta_i(\mathfrak K)=1
\qquad\text{for all }i\geq0,
\]
as required.
\end{proof}

Although Proposition~\ref{prop:K-order} is sufficient for the arguments in this subsection, it is useful to understand the module $\mathfrak K$ more explicitly. We therefore determine a generating set for $\mathfrak K$.

\begin{prop}
\label{prop:generators-for-K}
For $1\leq i\leq m$ and $2\leq j\leq\mu$, define
\[
g_{ij}:=(t_{1i}-1)(t_j-1).
\]
Then $\mathfrak K$ is generated by the elements
\[
g_{ij},
\qquad
1\leq i\leq m,\quad
2\leq j\leq\mu.
\]
\end{prop}
\begin{proof}
Although a purely algebraic proof is also possible, we use a
combinatorial interpretation of the map
\[
f\colon
H_0(\widehat\Sigma)
\longrightarrow
H_0(Y_C)\oplus H_0(Y_P).
\]
Let $\mathcal B$ be the complete bipartite graph whose two sets of
vertices are $\Gamma_C$ and $\Gamma_P$. For each
\[
[\mathbf n]\in\Gamma_C,
\qquad
[\mathbf r]\in\Gamma_P,
\]
let $e_{\mathbf n,\mathbf r}$ denote the edge joining
$[\mathbf n]$ to $[\mathbf r]$, oriented from the $\Gamma_C$-vertex
to the $\Gamma_P$-vertex.

By Lemma~\ref{lem:lifted-pieces-nonzero}, the vertices of $\mathcal B$
correspond to the connected components of $Y_C$ and $Y_P$, while the
edge $e_{\mathbf n,\mathbf r}$ corresponds to the connected component
\[
\mathbf t_1^{\mathbf n}\mathbf t_C^{\mathbf r}\widetilde\Sigma
\]
of $\widehat\Sigma$ contained in the intersection of the corresponding
components of $Y_C$ and $Y_P$. Consequently, there are natural
isomorphisms of $R$-modules
\[
C_1(\mathcal B)
\cong
H_0(\widehat\Sigma),
\qquad
C_0(\mathcal B)
\cong
H_0(Y_C)\oplus H_0(Y_P),
\]
under which the cellular boundary map
\[
\partial\colon
C_1(\mathcal B)
\longrightarrow
C_0(\mathcal B)
\]
agrees with $f$. Under the first identification, we identify the edge
$e_{\mathbf 0,\mathbf 0}$ with the corresponding generator of
$H_0(\widehat\Sigma)$, and hence with $1$ in
\[
H_0(\widehat\Sigma)
\cong
R/\langle T_1-1,A_1-1\rangle.
\]
Since the cellular boundary map agrees with $f$, we obtain
\[
\mathfrak K
=
\ker(f)
\cong
H_1(\mathcal B).
\]
Every cycle in a complete bipartite graph is a sum of $4$-cycles.
Moreover, the deck transformation group acts transitively on the edges
of $\mathcal B$, so after translating we may assume that a given
$4$-cycle contains the edge $e_{\mathbf 0,\mathbf 0}$. Let the other
two vertices be indexed by
\[
[\mathbf n]\in\Gamma_C,
\qquad
[\mathbf r]\in\Gamma_P.
\]
The corresponding cycle is
\begin{align*}
e_{\mathbf 0,\mathbf 0}
-e_{\mathbf n,\mathbf 0}
-e_{\mathbf 0,\mathbf r}
+e_{\mathbf n,\mathbf r}
&=
\bigl(
1-\mathbf t_1^{\mathbf n}
-\mathbf t_C^{\mathbf r}
+\mathbf t_1^{\mathbf n}\mathbf t_C^{\mathbf r}
\bigr)e_{\mathbf 0,\mathbf 0}
\\
&=
(1-\mathbf t_1^{\mathbf n})
(1-\mathbf t_C^{\mathbf r})
e_{\mathbf 0,\mathbf 0}.
\end{align*}
Now
\[
1-\mathbf t_1^{\mathbf n}
\in
\langle
1-t_{11},\ldots,1-t_{1m}
\rangle,
\]
and
\[
1-\mathbf t_C^{\mathbf r}
\in
\langle
1-t_2,\ldots,1-t_\mu
\rangle.
\]
Hence there exist $a_i,b_j\in R$ such that
\[
1-\mathbf t_1^{\mathbf n}
=
\sum_{i=1}^{m}a_i(1-t_{1i}),
\qquad
1-\mathbf t_C^{\mathbf r}
=
\sum_{j=2}^{\mu}b_j(1-t_j).
\]
Therefore,
\begin{align*}
(1-\mathbf t_1^{\mathbf n})
(1-\mathbf t_C^{\mathbf r})
e_{\mathbf 0,\mathbf 0}
&=
\sum_{i=1}^{m}
\sum_{j=2}^{\mu}
a_i b_j
(1-t_{1i})(1-t_j)
e_{\mathbf 0,\mathbf 0} \\
&=
\sum_{i=1}^{m}
\sum_{j=2}^{\mu}
a_i b_j
g_{ij}.
\end{align*}
Thus every $4$-cycle lies in the $R$-submodule generated by the
elements $g_{ij}$. Since the $4$-cycles generate $H_1(\mathcal B)$, the
elements $g_{ij}$ generate $\mathfrak K$.
\end{proof}

The preceding analysis of the module $\mathfrak K$ now yields the
first corollary towards Theorem~\ref{thm:higher-single-component-case1}. It gives the first of
the two divisibility statements needed for the proof. 

\begin{cor}
\label{cor:first-divisibility}
For each $k\geq 0$, the $k$-th Alexander polynomial of the satellite link satisfies
\[
\Delta_{L,k}
\;\Bigm|\;
\gcd_{\substack{\\[2pt] i+j=k}}
\!\Bigl(
    \Delta_{C,i}(T_1,\mathbf t_C)
    \;
    \Delta_{P',j}(\mathbf t_1,A_1)
\Bigr).
\]
\end{cor}
\begin{proof}
Applying Corollary~\ref{cor:orders-short-exact-sequence} to the short
exact sequence~\eqref{eq:MV-short-exact} gives
\[
\Delta_k\bigl(H_1(\widehat X_L)\bigr)
\;\Bigm|\;
\gcd_{\substack{\\[2pt] i_1+i_2=k}}
\!\Bigl(
    \Delta_{i_1}\bigl(H_1(Y_C)\oplus H_1(Y_P)\bigr)
    \;
    \Delta_{i_2}(\mathfrak K)
\Bigr).
\]
By Proposition~\ref{prop:K-order},
\[
\Delta_{i_2}(\mathfrak K)=1
\qquad\text{for all }i_2\geq0.
\]
Hence
\[
\Delta_k\bigl(H_1(\widehat X_L)\bigr)
\;\Bigm|\;
\gcd_{0\leq i_1\leq k}
\!\Bigl(
    \Delta_{i_1}\bigl(H_1(Y_C)\oplus H_1(Y_P)\bigr)
\Bigr).
\]
Recall that the orders form a decreasing sequence with respect to
divisibility. Therefore, we can simply write
\[
\Delta_k\bigl(H_1(\widehat X_L)\bigr)
\;\Bigm|\;
\Delta_k\bigl(H_1(Y_C)\oplus H_1(Y_P)\bigr).
\]
We now apply the direct-sum formula for orders from
Corollary~\ref{cor:orders-short-exact-sequence} to obtain
\[
\Delta_k\bigl(H_1(Y_C)\oplus H_1(Y_P)\bigr)
\doteq
\gcd_{\substack{\\[2pt] i+j=k}}
\!\Bigl(
    \Delta_i\bigl(H_1(Y_C)\bigr)
    \;
    \Delta_j\bigl(H_1(Y_P)\bigr)
\Bigr).
\]
Finally, Proposition~\ref{prop:homology-of-lifted-pieces}, together with
the behavior of Fitting ideals under extension of scalars, gives
\[
\Delta_i\bigl(H_1(Y_C)\bigr)
\doteq
\Delta_{C,i}(T_1,\mathbf t_C)
\]
and
\[
\Delta_j\bigl(H_1(Y_P)\bigr)
\doteq
\Delta_{P',j}(\mathbf t_1,A_1).
\]
Combining these identities with
\[
\Delta_k\bigl(H_1(\widehat X_L)\bigr)
\doteq
\Delta_{L,k}
\]
yields
\[
\Delta_{L,k}
\;\Bigm|\;
\gcd_{\substack{\\[2pt] i+j=k}}
\!\Bigl(
    \Delta_{C,i}(T_1,\mathbf t_C)
    \;
    \Delta_{P',j}(\mathbf t_1,A_1)
\Bigr),
\]
as required.
\end{proof}

To prove the reverse divisibility, we now use the relative
Mayer--Vietoris sequence~\eqref{eq:relative-MV-sequence}. To obtain a
short exact sequence from it, we first establish the following two facts.

\begin{prop}
\label{prop:relative-end-terms}
We have
\[
H_0(\widehat \Sigma,\widehat X_0)=0
\]
and the connecting homomorphism
\[
H_2(\widehat X_L,\widehat X_0)
\longrightarrow
H_1(\widehat \Sigma,\widehat X_0)
\]
is the zero map.
\end{prop}
\begin{proof}
Each connected component of $\widehat\Sigma$ contains a lift of $x_0$.
Therefore, the homomorphism
\[
H_0(\widehat X_0)
\longrightarrow
H_0(\widehat\Sigma)
\]
induced by inclusion is surjective. The long exact sequence of the pair
$(\widehat\Sigma,\widehat X_0)$ then gives
\[
H_0(\widehat\Sigma,\widehat X_0)=0.
\]

For the second assertion, since $\widehat X_0$ is zero-dimensional, the
long exact sequence of the pair $(\widehat X_L,\widehat X_0)$ gives a
natural isomorphism
\[
H_2(\widehat X_L)
\overset{\cong}{\longrightarrow}
H_2(\widehat X_L,\widehat X_0).
\]
Naturality of the Mayer--Vietoris connecting homomorphisms gives the
commutative diagram
\[
\begin{CD}
H_2(\widehat X_L)
@>>>
H_1(\widehat\Sigma)
\\
@V{\cong}VV
@VVV
\\
H_2(\widehat X_L,\widehat X_0)
@>>>
H_1(\widehat\Sigma,\widehat X_0),
\end{CD}
\]
where both vertical maps are induced by the inclusions of pairs. By
Equation~\eqref{eq:H1-lifted-torus},
\[
H_1(\widehat\Sigma)=0.
\]
Hence the upper horizontal map is zero. Since the left vertical map is
an isomorphism, commutativity implies that the lower horizontal map
\[
H_2(\widehat X_L,\widehat X_0)
\longrightarrow
H_1(\widehat\Sigma,\widehat X_0)
\]
is also zero.
\end{proof}

Combining Proposition~\ref{prop:relative-end-terms} with the relative
Mayer--Vietoris sequence~\eqref{eq:relative-MV-sequence}, we obtain the
short exact sequence
\begin{equation}
\label{eq:relative-short-sequence}
0
\longrightarrow
H_1(\widehat\Sigma,\widehat X_0)
\longrightarrow
H_1(Y_C,\widehat X_0)
\oplus
H_1(Y_P,\widehat X_0)
\longrightarrow
H_1(\widehat X_L,\widehat X_0)
\longrightarrow
0.
\end{equation}

The right-hand term of
\eqref{eq:relative-short-sequence} recovers the Alexander module of
$L$ by Proposition~\ref{prop:relative-Alexander-module}. For the same
reason, the middle terms are closely related to the corresponding
Alexander modules, but require one additional step. Let
\[
\widehat X_{0,C}\subseteq\widehat X_C
\qquad\text{and}\qquad
\widehat X_{0,P}\subseteq\widehat X_P
\]
denote the full preimages of $x_0$ in the maximal abelian covers
$\widehat X_C$ and $\widehat X_P$, respectively.

\begin{rem}
\label{rem:relative-modules-for-PC}
The argument of Proposition~\ref{prop:relative-Alexander-module} applies
equally to the relative homology modules
\[
H_1(\widehat X_C,\widehat X_{0,C})
\qquad\text{and}\qquad
H_1(\widehat X_P,\widehat X_{0,P}),
\]
viewed as modules over $R_C$ and $R_P$, respectively.
\end{rem}

\begin{prop}
\label{prop:relhomology-of-lifted-pieces}
There are natural isomorphisms of $R$-modules
\begin{align}
H_1(Y_C,\widehat X_0)
&\cong
H_1(\widehat X_C,\widehat X_{0,C})
\otimes_{R_C}R,
\label{eq:relative-H1-YC-extension}
\\
H_1(Y_P,\widehat X_0)
&\cong
H_1(\widehat X_P,\widehat X_{0,P})
\otimes_{R_P}R.
\label{eq:relative-H1-YP-extension}
\end{align}
\end{prop}

\begin{proof}
The proof is identical to that of
Proposition~\ref{prop:homology-of-lifted-pieces}.
\end{proof}

Finally, we compute the left-hand term in the short exact sequence
\eqref{eq:relative-short-sequence}.

\begin{prop}
\label{prop:relative-homology-of-torus}
There is a natural isomorphism of $R$-modules
\[
H_1(\widehat \Sigma,\widehat X_0)
\cong
\langle T_1-1,A_1-1\rangle
\subseteq R.
\]
\end{prop}

\begin{proof}
The long exact sequence of the pair
$(\widehat\Sigma,\widehat X_0)$ gives
\[
0
\longrightarrow
H_1(\widehat\Sigma,\widehat X_0)
\longrightarrow
H_0(\widehat X_0)
\longrightarrow
H_0(\widehat\Sigma)
\longrightarrow
0.
\]
Here we have used Equation~\eqref{eq:H1-lifted-torus}, which gives
$H_1(\widehat\Sigma)=0$, together with
Proposition~\ref{prop:relative-end-terms}, which gives
$H_0(\widehat\Sigma,\widehat X_0)=0$.

After choosing a lift of $x_0$, we have
\[
H_0(\widehat X_0)\cong R,
\]
while Equation~\eqref{eq:H0-lifted-Sigma} gives
\[
H_0(\widehat\Sigma)
\cong
R/\langle T_1-1,A_1-1\rangle.
\]
Under these identifications, the map
\[
H_0(\widehat X_0)\longrightarrow H_0(\widehat\Sigma)
\]
is the quotient homomorphism
\[
R\longrightarrow
R/\langle T_1-1,A_1-1\rangle.
\]
Its kernel is therefore
\[
\langle T_1-1,A_1-1\rangle,
\]
and exactness gives the desired isomorphism.
\end{proof}

\begin{cor}
\label{cor:Fitting-relative-torus}
The $R$-module $H_1(\widehat\Sigma,\widehat X_0)$ admits a presentation
\[
R
\xrightarrow{\;\psi\;}
R^2
\longrightarrow
H_1(\widehat\Sigma,\widehat X_0)
\longrightarrow
0,
\]
where
\[
\psi(1)
=
\bigl(A_1-1,-(T_1-1)\bigr).
\]
Consequently,
\[
E_0\bigl(H_1(\widehat\Sigma,\widehat X_0)\bigr)=0
\]
and
\[
E_1\bigl(H_1(\widehat\Sigma,\widehat X_0)\bigr)
=
\langle T_1-1,A_1-1\rangle.
\]
In particular,
\[
\Delta_1\bigl(H_1(\widehat\Sigma,\widehat X_0)\bigr)=1.
\]
\end{cor}
\begin{proof}
By Proposition~\ref{prop:relative-homology-of-torus}, it suffices to
consider the ideal
\[
I:=\langle T_1-1,A_1-1\rangle\subseteq R.
\]
The homomorphism
\[
R^2\longrightarrow I,
\qquad
(p,q)\longmapsto p(T_1-1)+q(A_1-1),
\]
is surjective.

Since $T_1-1$ and $A_1-1$ are relatively prime in the UFD $R$, the
kernel of this homomorphism is generated by
\[
\bigl(A_1-1,-(T_1-1)\bigr).
\]
Indeed, if
\[
p(T_1-1)+q(A_1-1)=0,
\]
then $T_1-1$ divides $q(A_1-1)$. Since $T_1-1$ and $A_1-1$ are
relatively prime, it follows that $T_1-1$ divides $q$. Writing
\[
q=r(T_1-1)
\]
and substituting into the relation gives
\[
p=-r(A_1-1).
\]
Thus
\[
(p,q)
=
-r\bigl(A_1-1,-(T_1-1)\bigr),
\]
which proves the asserted presentation.

By Remark~\ref{rem:adding-zero-relations}, we may add a zero row to the presentation matrix so that it has two rows and two columns. Therefore
\[
E_0\bigl(H_1(\widehat\Sigma,\widehat X_0)\bigr)=0,
\]
and
\[
E_1\bigl(H_1(\widehat\Sigma,\widehat X_0)\bigr)
=
\langle T_1-1,A_1-1\rangle.
\]
Finally, since
\[
\gcd(T_1-1,A_1-1)=1,
\]
we obtain
\[
\Delta_1\bigl(H_1(\widehat\Sigma,\widehat X_0)\bigr)=1.
\]
\end{proof}

We are now ready to prove the reverse divisibility.

\begin{cor}
\label{cor:reverse-divisibility}
For each $k\geq 0$, the $k$-th Alexander polynomial of the satellite
link satisfies
\[
\gcd_{\substack{\\[2pt] i+j=k}}
\!\Bigl(
    \Delta_{C,i}(T_1,\mathbf t_C)
    \;
    \Delta_{P',j}(\mathbf t_1,A_1)
\Bigr)
\;\Bigm|\;
\Delta_{L,k}.
\]
\end{cor}

\begin{proof}
For ease of notation, set
\begin{align*}
\mathfrak M_\Sigma
&:=
H_1(\widehat\Sigma,\widehat X_0),\\
\mathfrak M_C
&:=
H_1(Y_C,\widehat X_0),\\
\mathfrak M_P
&:=
H_1(Y_P,\widehat X_0),\\
\mathfrak M_L
&:=
H_1(\widehat X_L,\widehat X_0).
\end{align*}
The short exact sequence~\eqref{eq:relative-short-sequence}, together
with Corollary~\ref{cor:orders-short-exact-sequence}, gives, for every
$k\geq0$,
\[
\Delta_{k+2}(\mathfrak M_C\oplus\mathfrak M_P)
\;\Bigm|\;
\Delta_1(\mathfrak M_\Sigma)\,
\Delta_{k+1}(\mathfrak M_L).
\]
By Corollary~\ref{cor:Fitting-relative-torus},
\[
\Delta_1(\mathfrak M_\Sigma)=1.
\]
Hence
\[
\Delta_{k+2}(\mathfrak M_C\oplus\mathfrak M_P)
\;\Bigm|\;
\Delta_{k+1}(\mathfrak M_L).
\]
Proposition~\ref{prop:relative-Alexander-module} gives
\[
\Delta_{k+1}(\mathfrak M_L)
\doteq
\Delta_{L,k}.
\]
On the other hand, the direct-sum formula for orders from
Corollary~\ref{cor:orders-short-exact-sequence} gives
\[
\Delta_{k+2}(\mathfrak M_C\oplus\mathfrak M_P)
\doteq
\gcd_{\substack{\\[2pt] i+j=k+2}}
\!\Bigl(
    \Delta_i(\mathfrak M_C)
    \;
    \Delta_j(\mathfrak M_P)
\Bigr).
\]
By Proposition~\ref{prop:relhomology-of-lifted-pieces},
Remark~\ref{rem:relative-modules-for-PC}, and
Lemma~\ref{lem:fitting-base-change},
\[
\Delta_i(\mathfrak M_C)
\doteq
\Delta_{C,i-1}(T_1,\mathbf t_C)
\qquad\text{and}\qquad
\Delta_j(\mathfrak M_P)
\doteq
\Delta_{P',j-1}(\mathbf t_1,A_1).
\]
Therefore,
\[
\Delta_{k+2}(\mathfrak M_C\oplus\mathfrak M_P)
\doteq
\gcd_{\substack{\\[2pt] i+j=k}}
\!\Bigl(
    \Delta_{C,i}(T_1,\mathbf t_C)
    \;
    \Delta_{P',j}(\mathbf t_1,A_1)
\Bigr).
\]
Combining the preceding identities yields
\[
\gcd_{\substack{\\[2pt] i+j=k}}
\!\Bigl(
    \Delta_{C,i}(T_1,\mathbf t_C)
    \;
    \Delta_{P',j}(\mathbf t_1,A_1)
\Bigr)
\;\Bigm|\;
\Delta_{L,k},
\]
as required.
\end{proof}

We are now ready to complete the proof of
Theorem~\ref{thm:higher-single-component-case1}.

\begin{proof}[Proof of Theorem~\ref{thm:higher-single-component-case1}]
The result follows immediately by combining
Corollaries~\ref{cor:first-divisibility} and
\ref{cor:reverse-divisibility}.
\end{proof}

\subsection{Winding and Linking Number Zero}
\label{subsec:winding-linking-zero}\hfill

\noindent
Throughout this subsection, we assume that
\[
\mathbf w_1=\mathbf 0
\qquad\text{and}\qquad
\boldsymbol\ell_1=\mathbf 0.
\]
Our main goal in this subsection is to prove
Theorem~\ref{thm:higher-single-component-case4}.

The argument follows the same general strategy as in the previous
subsection, but in this case we obtain only the two divisibility
statements appearing in
Theorem~\ref{thm:higher-single-component-case4}, rather than an exact
formula. As we will see, this loss of equality ultimately comes from
Corollary~\ref{cor:traldi-quotient-bounds}. We continue to use the
notation introduced in the preceding subsection, but the relevant
covering spaces now have a different form.

As before, we begin with the absolute Mayer--Vietoris sequence
\eqref{eq:absolute-MV-sequence}. We first analyze the connected
components of the lifted pieces. For this, we recall the special abelian
cover introduced in Definition~\ref{defn:special-abelian-cover} and use
its analogue for the pattern piece.

\begin{defn}\label{defn:special-abelian-cover-companion-pattern}
Let
\[
Z_C\longrightarrow X_C
\]
denote the special abelian cover of $X_C$ obtained by forgetting the
first component, as in Definition~\ref{defn:special-abelian-cover}.
Thus it corresponds to the kernel of
\[
\pi_1(X_C,x_0)
\xrightarrow{\ \operatorname{ab}\ }
H_1(X_C;\mathbb Z)
\cong
\mathbb Z^\mu
\xrightarrow{\ h_C\ }
\mathbb Z^{\mu-1},
\]
where $h_C$ is projection onto the last $\mu-1$ coordinates.

Similarly, let
\[
Z_P\longrightarrow X_P
\]
denote the cover corresponding to the kernel of
\[
\pi_1(X_P,x_0)
\xrightarrow{\ \operatorname{ab}\ }
H_1(X_P;\mathbb Z)
\cong
\mathbb Z^{m+1}
\xrightarrow{\ h_P\ }
\mathbb Z^m,
\]
where $h_P$ is projection onto the first $m$ coordinates. Equivalently,
viewing $X_P$ as the exterior of the augmented link $P'=P\cup\rho$,
the cover $Z_P$ is the special abelian cover obtained by forgetting the augmented component $\rho$.
\end{defn}

We can now formulate the analogue of
Lemma~\ref{lem:lifted-pieces-nonzero} for the present case. We first
record how the indexing sets of
Definition~\ref{defn:component-indexing-sets} simplify under our
assumptions.

\begin{rem}
\label{rem:indexing-sets-zero-zero}
Since
\[
\mathbf w_1=\mathbf 0
\qquad\text{and}\qquad
\boldsymbol\ell_1=\mathbf 0,
\]
the indexing sets of Definition~\ref{defn:component-indexing-sets}
reduce to
\[
\Gamma_C=\mathbb Z^{m},
\qquad
\Gamma_P=\mathbb Z^{\mu-1},
\qquad
\Gamma_\Sigma
=
\mathbb Z^{m}\oplus\mathbb Z^{\mu-1}.
\]
\end{rem}

\begin{lem}
\label{lem:lifted-pieces-zero-zero}
There are identifications
\begin{align}
Y_C
&=
\bigsqcup_{\mathbf n\in\Gamma_C}
\mathbf t_1^{\mathbf n}Z_C,
\label{eq:YC-components-zero-zero}
\\
Y_P
&=
\bigsqcup_{\mathbf r\in\Gamma_P}
\mathbf t_C^{\mathbf r}Z_P,
\label{eq:YP-components-zero-zero}
\\
\widehat\Sigma
&=
\bigsqcup_{(\mathbf n,\mathbf r)\in\Gamma_\Sigma}
\mathbf t_1^{\mathbf n}\mathbf t_C^{\mathbf r}\Sigma.
\label{eq:Sigma-components-zero-zero}
\end{align}
\end{lem}

\begin{proof}
The argument is identical to that of Lemma~\ref{lem:lifted-pieces-nonzero};
in particular, the connected components are determined by the image of
the corresponding homomorphism on first homology. In the present case,
however, the map
\[
\iota_*\colon
\mathbb Z^\mu
\longrightarrow
\mathbb Z^{m}\oplus\mathbb Z^{\mu-1},
\qquad
(a_1,\ldots,a_\mu)
\longmapsto
(a_1\mathbf w_1,a_2,\ldots,a_\mu)
\]
is no longer injective, since $\mathbf w_1=\mathbf 0$. Consequently,
the connected component of $Y_C$ is no longer the maximal abelian cover
$\widehat X_C$, but rather the cover $Z_C$ defined above. The same
argument applies to $Y_P$ and $\widehat\Sigma$, with the corresponding
maps and covers.
\end{proof}

\begin{rem}
\label{rem:H0-lifted-pieces-zero-zero}
As in Remark~\ref{rem:H0-lifted-pieces}, the description of the
connected components in Lemma~\ref{lem:lifted-pieces-zero-zero}
determines the zeroth homology of the lifted pieces. Since
\[
T_1=1
\qquad\text{and}\qquad
A_1=1
\]
under the present assumptions, we obtain
\begin{align}
H_0(Y_C)
&\cong
R\big/
\langle t_2-1,\ldots,t_\mu-1\rangle,
\label{eq:H0-YC-zero-zero}
\\
H_0(Y_P)
&\cong
R\big/
\langle t_{11}-1,\ldots,t_{1m}-1\rangle,
\label{eq:H0-YP-zero-zero}
\\
H_0(\widehat\Sigma)
&\cong
R.
\label{eq:H0-Sigma-zero-zero}
\end{align}
\end{rem}

\begin{rem}
\label{rem:relative-end-terms-zero-zero}
Although the covering spaces appearing in the present case differ from
those in the previous subsection, the argument of
Proposition~\ref{prop:relative-end-terms} still applies to the relative
zeroth homology. In particular, every connected component of
$\widehat\Sigma$ contains a lift of $x_0$, and hence
\[
H_0(\widehat\Sigma,\widehat X_0)=0.
\]
\end{rem}

We again extract short exact sequences from the absolute and relative
Mayer--Vietoris sequences
\eqref{eq:absolute-MV-sequence} and
\eqref{eq:relative-MV-sequence}. Unlike in the previous subsection,
the relevant end terms do not all vanish, so the resulting short exact
sequences take a slightly different form.

\begin{cons}[Short exact sequences]
\label{cons:zero-zero-short-exact-sequences}
Let
\[
f_1\colon
H_1(\widehat\Sigma)
\longrightarrow
H_1(Y_C)\oplus H_1(Y_P)
\]
denote the corresponding map in the absolute Mayer--Vietoris sequence
\eqref{eq:absolute-MV-sequence}. Using
Definition~\ref{defn:K-module}, exactness gives
\begin{equation}
\label{eq:00-case-MV}
0
\longrightarrow
\frac{H_1(Y_C)\oplus H_1(Y_P)}
     {\operatorname{im}(f_1)}
\longrightarrow
H_1(\widehat X_L)
\longrightarrow
\mathfrak K
\longrightarrow
0.
\end{equation}
Similarly, let
\[
f_1^{\mathrm{rel}}\colon
H_1(\widehat\Sigma,\widehat X_0)
\longrightarrow
H_1(Y_C,\widehat X_0)\oplus H_1(Y_P,\widehat X_0)
\]
denote the corresponding map in the relative Mayer--Vietoris sequence
\eqref{eq:relative-MV-sequence}. By
Remark~\ref{rem:relative-end-terms-zero-zero}, exactness gives
\begin{equation}
\label{eq:00-case-MV-relative}
0
\longrightarrow
\operatorname{im}(f_1^{\mathrm{rel}})
\longrightarrow
H_1(Y_C,\widehat X_0)\oplus H_1(Y_P,\widehat X_0)
\longrightarrow
H_1(\widehat X_L,\widehat X_0)
\longrightarrow
0.
\end{equation}
\end{cons}

\begin{rem}
\label{rem:K-module-zero-zero}
As in Definition~\ref{defn:K-module}, the module $\mathfrak K$ appearing
in \eqref{eq:00-case-MV} is the kernel of the map
\[
H_0(\widehat\Sigma)
\longrightarrow
H_0(Y_C)\oplus H_0(Y_P)
\]
in the absolute Mayer--Vietoris sequence. However, the description in
Remark~\ref{rem:H0-lifted-pieces-zero-zero} shows that $\mathfrak K$
is different from the module appearing in the previous subsection. In
particular, it is now a submodule, and indeed an ideal, of
\[
H_0(\widehat\Sigma)\cong R,
\]
and hence is torsion-free. Indeed, as we
will see in Proposition~\ref{prop:module-k-for-00-case}, the module
$\mathfrak K$ has positive rank. Nevertheless, it admits the same
generating set as in Proposition~\ref{prop:generators-for-K}.
\end{rem}

We next analyze the first homology of the lifted pieces.

\begin{prop}
\label{prop:homology-of-lifted-pieces-zero-zero}
There are natural isomorphisms of $R$-modules
\begin{align}
H_1(\widehat\Sigma)
&\cong R^2,
\label{eq:H1-Sigma-zero-zero}
\\
H_1(Y_C)
&\cong
H_1(Z_C)\otimes_{R_C}R,
\label{eq:H1-YC-zero-zero}
\\
H_1(Y_P)
&\cong
H_1(Z_P)\otimes_{R_P}R.
\label{eq:H1-YP-zero-zero}
\end{align}
\end{prop}

\begin{proof}
The argument is identical to that of
Proposition~\ref{prop:homology-of-lifted-pieces}, using the description
of the lifted pieces in Lemma~\ref{lem:lifted-pieces-zero-zero}.
\end{proof}

We now compute the orders of $\mathfrak K$. As suggested by
Remark~\ref{rem:K-module-zero-zero}, its behavior differs from that in
the previous subsection.

\begin{prop}
\label{prop:module-k-for-00-case}
The module $\mathfrak K$ satisfies
\[
\Delta_0(\mathfrak K)=0
\qquad\text{and}\qquad
\Delta_i(\mathfrak K)=1
\quad\text{for all }i\geq1.
\]
\end{prop}

\begin{proof}
As suggested in Remark \ref{rem:K-module-zero-zero}, $\mathfrak K$ is a
submodule of a free $R$-module of rank one and is generated by the
elements
\[
g_{ij}=(t_{1i}-1)(t_j-1),
\qquad
1\leq i\leq m,\quad 2\leq j\leq\mu.
\]
In particular, $\mathfrak K\neq0$. Hence, by
Lemma~\ref{lem:rank-of-submodules-of-free-modules},
\[
\operatorname{rank}_R(\mathfrak K)>0.
\]
Since $\mathfrak K\subseteq H_0(\widehat\Sigma)\cong R$, we also have
\[
\operatorname{rank}_R(\mathfrak K)\leq1.
\]
Therefore,
\[
\operatorname{rank}_R(\mathfrak K)=1.
\]
and consequently
\[
\Delta_0(\mathfrak K)=0.
\]
It remains to determine the higher orders. Since the orders form a
decreasing sequence with respect to divisibility, it suffices to prove
that
\[
\Delta_1(\mathfrak K)=1.
\]
If $m=1$ and $\mu=2$, then $\mathfrak K$ is a principal ideal of $R$,
As a result, it is a free $R$-module of rank one, and hence
\[
\Delta_1(\mathfrak K)=1.
\]
We may therefore assume that $(m,\mu)\neq(1,2)$.

Let
\[
\left\{
x_{ij}
\ \middle|\
1\leq i\leq m,\;
2\leq j\leq\mu
\right\}
\]
be the standard basis of $R^{m(\mu-1)}$, and define
\[
g\colon
R^{m(\mu-1)}
\longrightarrow
\mathfrak K
\]
by
\[
g(x_{ij})=g_{ij}.
\]
Since the elements $g_{ij}$ generate $\mathfrak K$, we have a short
exact sequence
\[
0
\longrightarrow
\ker(g)
\longrightarrow
R^{m(\mu-1)}
\xrightarrow{\ g\ }
\mathfrak K
\longrightarrow
0.
\]

Since $R$ is Noetherian, $\ker(g)$ is finitely generated. Choose a
finite generating set $\{y_p\}_{p\in I}$ for $\ker(g)$. This gives a
presentation
\[
R^{|I|}
\xrightarrow{\ \phi\ }
R^{m(\mu-1)}
\xrightarrow{\ g\ }
\mathfrak K
\longrightarrow
0.
\]
The generators $g_{ij}$ satisfy the relations
\[
(t_k-1)g_{ij}
=
(t_j-1)g_{ik}
\]
and
\[
(t_{1k}-1)g_{ij}
=
(t_{1i}-1)g_{kj}.
\]
Accordingly, we may enlarge the chosen generating set of $\ker(g)$ by
elements $z_{i,jk}$ and $z_{ik,j}$ corresponding to these relations.
We obtain a presentation
\[
R^{|I|+m\binom{\mu-1}{2}+\binom{m}{2}(\mu-1)}
\xrightarrow{\ \widetilde\phi\ }
R^{m(\mu-1)}
\xrightarrow{\ g\ }
\mathfrak K
\longrightarrow
0,
\]
where
\[
\widetilde\phi(y_p)=\phi(y_p),
\]
\[
\widetilde\phi(z_{i,jk})
=
-(t_k-1)x_{ij}+(t_j-1)x_{ik},
\]
and
\[
\widetilde\phi(z_{ik,j})
=
-(t_{1k}-1)x_{ij}+(t_{1i}-1)x_{kj}.
\]
The first family of relations preserves the first index and varies the
second, whereas the second family preserves the second index and varies
the first. We will refer to $z_{i,jk}$ as \emph{relating} $x_{ij}$ and
$x_{ik}$, and to $z_{ik,j}$ as \emph{relating} $x_{ij}$ and $x_{kj}$.

We now exhibit two $(m(\mu-1)-1)\times(m(\mu-1)-1)$ minors of a
presentation matrix whose determinants are relatively prime.

First, let $\widetilde M_{11,2}$ be the
$(m(\mu-1)-1)\times m(\mu-1)$ matrix obtained from the relations
$z_{1i,2}$, for $2\leq i\leq m$, together with the relations
$z_{i,2j}$, for $1\leq i\leq m$ and $3\leq j\leq\mu$.
For example, when $m=3$ and $\mu=4$, this matrix is
\begin{center}
{\small
\setlength{\arraycolsep}{3pt}
$\begin{array}{c|ccccccccc}
 & x_{12} & x_{22} & x_{32} & x_{13} & x_{23} & x_{33} & x_{14} & x_{24} & x_{34} \\ \hline
z_{12,2}
& -(t_{12}-1) & t_{11}-1 & 0 & 0 & 0 & 0 & 0 & 0 & 0 \\
z_{13,2}
& -(t_{13}-1) & 0 & t_{11}-1 & 0 & 0 & 0 & 0 & 0 & 0 \\
z_{1,23}
& -(t_3-1) & 0 & 0 & t_2-1 & 0 & 0 & 0 & 0 & 0 \\
z_{2,23}
& 0 & -(t_3-1) & 0 & 0 & t_2-1 & 0 & 0 & 0 & 0 \\
z_{3,23}
& 0 & 0 & -(t_3-1) & 0 & 0 & t_2-1 & 0 & 0 & 0 \\
z_{1,24}
& -(t_4-1) & 0 & 0 & 0 & 0 & 0 & t_2-1 & 0 & 0 \\
z_{2,24}
& 0 & -(t_4-1) & 0 & 0 & 0 & 0 & 0 & t_2-1 & 0 \\
z_{3,24}
& 0 & 0 & -(t_4-1) & 0 & 0 & 0 & 0 & 0 & t_2-1
\end{array}$.}
\end{center}
We order the columns, corresponding to the $x_{ij}$, lexicographically
by the pairs $(j,i)$; explicitly,
\begin{equation}
\label{eq:column-order-M11-2}
\underbrace{x_{12},x_{22},\ldots,x_{m2}}_{\text{$j=2$}}
\ ;\ 
\underbrace{x_{13},x_{23},\ldots,x_{m3}}_{\text{$j=3$}}
\ ;\ \cdots\ ;\
\underbrace{x_{1\mu},x_{2\mu},\ldots,x_{m\mu}}_{\text{$j=\mu$}}.
\end{equation}
Similarly, we order the rows in corresponding blocks. The first block
consists of the $m-1$ relations connecting the remaining generators in
the $j=2$ block to $x_{12}$. For each $j\geq3$, the corresponding block
consists of the $m$ relations connecting the $i$-th generator in the
$j$-block to the $i$-th generator in the $j=2$ block. Explicitly, the
ordering is
\[
\underbrace{
z_{12,2},z_{13,2},\ldots,z_{1m,2}
}_{\text{$j=2$}}
\ ;\
\underbrace{
z_{1,23},z_{2,23},\ldots,z_{m,23}
}_{\text{$j=3$}}
\ ;\ \cdots\ ;\
\underbrace{
z_{1,2\mu},z_{2,2\mu},\ldots,z_{m,2\mu}
}_{\text{$j=\mu$}}.
\]
Let $M_{11,2}$ be the square matrix obtained from
$\widetilde M_{11,2}$ by deleting the column corresponding to
$x_{12}$. With these orderings, $M_{11,2}$ is lower triangular. Its
diagonal entries consist of $m-1$ copies of $(t_{11}-1)$ and
$m(\mu-2)$ copies of $(t_2-1)$. Hence
\begin{equation}
\label{eq:det-M11-2}
\det(M_{11,2})
=
(t_{11}-1)^{m-1}
(t_2-1)^{m(\mu-2)}.
\end{equation}

The construction of the second matrix is analogous. Starting from the
ordering in Equation~\eqref{eq:column-order-M11-2}, we first exchange
the first two elements $x_{13}$ and $x_{23}$ in the $j=3$ block, and
then exchange the $j=2$ and $j=3$ blocks. Thus, the resulting ordering
of the generators is
\begin{equation}
\label{eq:column-order-M12-3}
\underbrace{x_{23},x_{13},x_{33},\ldots,x_{m3}}_{\text{$j=3$}}
\ ;\
\underbrace{x_{12},x_{22},\ldots,x_{m2}}_{\text{$j=2$}}
\ ;\
\underbrace{x_{14},x_{24},\ldots,x_{m4}}_{\text{$j=4$}}
\ ;\ \cdots\ ;\
\underbrace{x_{1\mu},x_{2\mu},\ldots,x_{m\mu}}_{\text{$j=\mu$}}.
\end{equation}

We choose the relations according to the same rule as for
$\widetilde M_{11,2}$. The first $m-1$ relations relate the remaining
generators in the first block to $x_{23}$, while the relations in each
subsequent block relate its elements, position by position, to the
corresponding elements of the first block. Accordingly, we order the
rows as
\[
\underbrace{
z_{12,3},z_{23,3},z_{24,3},\ldots,z_{2m,3}
}_{\text{$j=3$}}
\ ;\
\underbrace{
z_{1,23},z_{2,23},\ldots,z_{m,23}
}_{\text{$j=2$}}
\ ;\
\underbrace{
z_{1,34},z_{2,34},\ldots,z_{m,34}
}_{\text{$j=4$}}
\ ;\ \cdots\ ;\
\underbrace{
z_{1,3\mu},z_{2,3\mu},\ldots,z_{m,3\mu}
}_{\text{$j=\mu$}}.
\]
When $m=1$ or $\mu=2$, the same construction applies, with any
unavailable exchange omitted.

Let $\widetilde M_{12,3}$ be the matrix determined by these relations,
and let $M_{12,3}$ be the square matrix obtained by deleting the column
corresponding to $x_{23}$. With these orderings, $M_{12,3}$ is lower
triangular. Its diagonal entries consist of $m-1$ copies of
$\pm (t_{12}-1)$ and $m(\mu-2)$ copies of $\pm (t_3-1)$. Hence
\begin{equation}
\label{eq:det-M12-3}
\det(M_{12,3})
=
\pm (t_{12}-1)^{m-1}
(t_3-1)^{m(\mu-2)}.
\end{equation}

Both determinants are $(m(\mu-1)-1)$-minors of a presentation matrix
for $\mathfrak K$, and hence
\[
\det(M_{11,2}),
\det(M_{12,3})
\in
E_1(\mathfrak K).
\]
Therefore,
\[
\Delta_1(\mathfrak K)
\;\Bigm|\;
\gcd\bigl(
\det(M_{11,2}),
\det(M_{12,3})
\bigr).
\]
By Equations~\eqref{eq:det-M11-2} and
\eqref{eq:det-M12-3}, the two determinants are relatively prime, so
\[
\gcd\bigl(
\det(M_{11,2}),
\det(M_{12,3})
\bigr)=1.
\]
Consequently,
\[
\Delta_1(\mathfrak K)=1.
\]
It follows that
\[
\Delta_i(\mathfrak K)=1
\qquad
\text{for every }i\geq1,
\]
which completes the proof.
\end{proof}

We now turn our attention to the left-hand term in the short exact
sequence~\eqref{eq:00-case-MV}. We first recall some facts that will be
used in its analysis.

\begin{rem}
\label{rem:specialization-zero-zero}
We use the notation from Remark~\ref{rem:change-of-variables} together
with the identifications established in
Proposition~\ref{prop:homology-of-lifted-pieces-zero-zero}. In the present
case,
\[
\phi_C(s_1)=1
\qquad\text{and}\qquad
\phi_P(z_1)=1.
\]
Let
\[
\psi_C\colon
R_C\twoheadrightarrow
\mathbb Z[t_2^{\pm1},\ldots,t_\mu^{\pm1}]
\]
and
\[
\psi_P\colon
R_P\twoheadrightarrow
\mathbb Z[t_{11}^{\pm1},\ldots,t_{1m}^{\pm1}]
\]
denote the homomorphisms obtained by setting $s_1=1$ and $z_1=1$,
respectively. Then, from the proof of
Lemma~\ref{lem:specialization-abelian-cover}, we have
\[
E_i\bigl(H_1(Z_C)\bigr)
=
\psi_C\bigl(E_i(H_1(\widehat X_C))\bigr)
\]
and
\[
E_i\bigl(H_1(Z_P)\bigr)
=
\psi_P\bigl(E_i(H_1(\widehat X_P))\bigr).
\]
\end{rem}
\begin{lem}
\label{lem:quotient-module-zero-zero}
Let
\[
\mathfrak B
:=
\frac{H_1(Y_C)\oplus H_1(Y_P)}
{\operatorname{im}(f_1)}.
\]
Then, for every $k\geq 0$,
\[
\Delta_k(\mathfrak B)
\;\Bigm|\;
\gcd_{\substack{\\[2pt] i+j=k-1}}
\!\Bigl(
    \Delta_{C',i}(\mathbf t_C)
    \;
    \Delta_{P,j}(\mathbf t_1)
\Bigr).
\]
\end{lem}
\begin{proof}
Let $M_C$ and $M_P$ be presentation matrices for $H_1(Y_C)$ and
$H_1(Y_P)$, respectively, with $n_C$ and $n_P$ columns. By
Proposition~\ref{prop:homology-of-lifted-pieces-zero-zero},
\[
H_1(\widehat\Sigma)\cong R^2.
\]
Choose generators $y,y'\in H_1(\widehat\Sigma)$ corresponding to the
lifts of the meridian and longitude of $\Sigma=\partial N(C_1)$,
respectively. Then $\operatorname{im}(f_1)$ is generated by
$f_1(y)$ and $f_1(y')$, either of which may vanish.

Let
\[
\operatorname{pr}_C\colon
H_1(Y_C)\oplus H_1(Y_P)\longrightarrow H_1(Y_C)
\]
and
\[
\operatorname{pr}_P\colon
H_1(Y_C)\oplus H_1(Y_P)\longrightarrow H_1(Y_P)
\]
denote the natural projections. Write
\[
r_C:=\operatorname{pr}_C(f_1(y)),
\qquad
r_P:=\operatorname{pr}_P(f_1(y)),
\]
and
\[
r_C':=\operatorname{pr}_C(f_1(y')),
\qquad
r_P':=\operatorname{pr}_P(f_1(y')).
\]
We regard these elements as row vectors with respect to the chosen
generators of the corresponding presentation matrices. A presentation
matrix for $\mathfrak B$ is therefore
\[
M_{\mathfrak B}
=
\begin{pmatrix}
M_C & 0 \\
r_C & r_P \\
r_C' & r_P' \\
0 & M_P
\end{pmatrix}.
\]
Thus $\Delta_k(\mathfrak B)$ is the greatest common divisor of the
$(n_C+n_P-k)$-minors of $M_{\mathfrak B}$.

Define
\[
f_C
:=
\left.
\operatorname{pr}_C\circ f_1
\right|_{\langle y\rangle}
\]
and
\[
f_P
:=
\left.
\operatorname{pr}_P\circ f_1
\right|_{\langle y'\rangle}.
\]
We claim that
\begin{equation}
\label{eq:inclusion-of-subideals}
\begin{aligned}
&E_k\left(
H_1(Y_C)/\operatorname{im}(f_C)
\oplus
H_1(Y_P)
\right)
\\
&\qquad+
E_k\left(
H_1(Y_C)
\oplus
H_1(Y_P)/\operatorname{im}(f_P)
\right)
\subseteq
E_k(\mathfrak B).
\end{aligned}
\end{equation}
For the first inclusion, observe that
\[
M_1
=
\begin{pmatrix}
M_C & 0 \\
r_C & 0 \\
0 & M_P
\end{pmatrix}
\]
is a presentation matrix for
\[
H_1(Y_C)/\operatorname{im}(f_C)
\oplus
H_1(Y_P).
\]
Let $p$ be a nonzero $(n_C+n_P-k)$-minor of $M_1$, arising from a
submatrix
\[
M_p
=
\begin{pmatrix}
\widetilde M_C & 0 \\
\widetilde r_C & 0 \\
0 & \widetilde M_P
\end{pmatrix},
\]
where the tildes indicate that some of the corresponding rows and
columns may have been deleted. Since $p\neq0$, the selected rows and
columns must balance within each diagonal block, so the two diagonal
blocks are both squares.  We refer to these square diagonal blocks as the
$C$-block and the $P$-block, respectively.

Choose the same rows and columns in $M_{\mathfrak B}$. The
corresponding submatrix is
\[
M_{\mathfrak B,p}
=
\begin{pmatrix}
\widetilde M_C & 0 \\
\widetilde r_C & \widetilde r_P \\
0 & \widetilde M_P
\end{pmatrix}.
\]
We claim that
\[
\det(M_{\mathfrak B,p})=p.
\]
If the $(r_C,r_P)$-row is not selected, then the equality is immediate.
Otherwise, $\widetilde M_C$ has size
$(a_C-1)\times a_C$, since adjoining the row $\widetilde r_C$ makes
the corresponding $C$-block square.

We now consider the Leibniz expansion of
$\det(M_{\mathfrak B,p})$. It suffices to show that no nonzero term
contains an entry of $\widetilde r_P$. Indeed, if a term contains an
entry of $\widetilde r_P$, then the $(r_C,r_P)$-row is paired with a
column in the $P$-block. The $a_C$ columns in the $C$-block must then
be paired with only the $a_C-1$ rows of $\widetilde M_C$ together with
rows from the $P$-block. Hence at least one $C$-column is paired with a
row from the $P$-block, where the corresponding entry is zero.
Therefore every term involving an entry of $\widetilde r_P$ vanishes.

It follows that
\[
\det(M_{\mathfrak B,p})
=
\det(M_p)
=
p.\]
Therefore,
\[
E_k\left(
H_1(Y_C)/\operatorname{im}(f_C)
\oplus
H_1(Y_P)
\right)
\subseteq
E_k(\mathfrak B).
\]

The same argument, using the row $(r_C',r_P')$, gives
\[
E_k\left(
H_1(Y_C)
\oplus
H_1(Y_P)/\operatorname{im}(f_P)
\right)
\subseteq
E_k(\mathfrak B),
\]
and hence proves~\eqref{eq:inclusion-of-subideals}.

We now pass from the inclusion of Fitting ideals to orders. Combining Lemma~\ref{lem:quotient-by-torus-image} with Proposition~\ref{prop:homology-of-lifted-pieces-zero-zero}, we obtain
\[
H_1(Y_C)/\operatorname{im}(f_C)
\cong
H_1(\widehat X_{C'})
\otimes R,
\]
and
\[
H_1(Y_P)/\operatorname{im}(f_P)
\cong
H_1(\widehat X_P)\otimes R,
\]
Here the longitude of $\partial N(C_1)$ corresponds, on the pattern
side, to the meridian of the augmented component $\rho$ of
$P'=P\cup\rho$.

Taking greatest common divisors in
\eqref{eq:inclusion-of-subideals} and using the direct-sum formula for
orders gives
\[
\Delta_k(\mathfrak B)
\;\Bigm|\;
\gcd\left(
\gcd_{\substack{i+j=k}}
\!\Bigl(
\Delta_{C',i}(\mathbf t_C)
\;
\Delta_j\bigl(H_1(Y_P)\bigr)
\Bigr),
\;
\gcd_{\substack{i+j=k}}
\!\Bigl(
\Delta_i\bigl(H_1(Y_C)\bigr)
\;
\Delta_{P,j}(\mathbf t_1)
\Bigr)
\right).
\]
By Proposition~\ref{prop:homology-of-lifted-pieces-zero-zero},
Remark~\ref{rem:specialization-zero-zero}, and
Corollary~\ref{cor:traldi-quotient-bounds},
\[
\Delta_j\bigl(H_1(Y_P)\bigr)
\;\Bigm|\;
\Delta_{P,j-1}(\mathbf t_1)
\]
and
\[
\Delta_i\bigl(H_1(Y_C)\bigr)
\;\Bigm|\;
\Delta_{C',i-1}(\mathbf t_C).
\]
Substituting these divisibilities into the preceding expression and
reindexing yields
\[
\Delta_k(\mathfrak B)
\;\Bigm|\;
\gcd_{\substack{\\[2pt] i+j=k-1}}
\!\Bigl(
\Delta_{C',i}(\mathbf t_C)
\;
\Delta_{P,j}(\mathbf t_1)
\Bigr),
\]
as required.
\end{proof}

The preceding analysis and Proposition~\ref{prop:module-k-for-00-case}
give the first divisibility in Theorem~\ref{thm:higher-single-component-case4}.

\begin{cor}
\label{cor:first-divisibility-zero-zero}
For each $k\geq0$, the $k$-th Alexander polynomial of the satellite
link satisfies
\[
\Delta_{L,k}
\;\Bigm|\;
\gcd_{\substack{\\[2pt] i+j=k-2}}
\!\Bigl(
    \Delta_{C',i}(\mathbf t_C)
    \;
    \Delta_{P,j}(\mathbf t_1)
\Bigr).
\]
\end{cor}
\begin{proof}
Recall from \eqref{eq:00-case-MV} that we have the short exact sequence
\[
0
\longrightarrow
\mathfrak B
\longrightarrow
H_1(\widehat X_L)
\longrightarrow
\mathfrak K
\longrightarrow
0.
\]
By Corollary~\ref{cor:orders-short-exact-sequence},
\[
\Delta_{L,k}
\;\Bigm|\;
\gcd_{\substack{\\[2pt] i+j=k}}
\!\Bigl(
\Delta_i(\mathfrak B)\,
\Delta_j(\mathfrak K)
\Bigr).
\]
By Proposition~\ref{prop:module-k-for-00-case}, we obtain
\[
\Delta_{L,k}
\;\Bigm|\;
\Delta_{k-1}(\mathfrak B).
\]
Applying Lemma~\ref{lem:quotient-module-zero-zero} to
$\Delta_{k-1}(\mathfrak B)$ gives
\[
\Delta_{L,k}
\;\Bigm|\;
\gcd_{\substack{\\[2pt] i+j=k-2}}
\!\Bigl(
\Delta_{C',i}(\mathbf t_C)\,
\Delta_{P,j}(\mathbf t_1)
\Bigr),
\]
as required.
\end{proof}

We next use the relative short exact sequence
\eqref{eq:00-case-MV-relative} to obtain the opposite divisibility. To
simplify the statement, we use the notation $D_{C,i}$ and $D_{P,i}$
from Definition~\ref{defn:DC-DP}.

\begin{prop}
\label{prop:reverse-divisibility-zero-zero}
For each $k\geq0$, the $k$-th Alexander polynomial of the satellite
link satisfies
\[
\gcd_{\substack{\\[2pt] i+j=k+1}}
\!\Bigl(
D_{C,i}(\mathbf t_C)\,D_{P,j}(\mathbf t_1)
\Bigr)
\;\Bigm|\;
\Delta_{L,k},
\]
\end{prop}

\begin{proof}
Recall the short exact sequence~\eqref{eq:00-case-MV-relative},
\[
0
\longrightarrow
\operatorname{im}(f_1^{\mathrm{rel}})
\longrightarrow
H_1(Y_C,\widehat X_0)\oplus H_1(Y_P,\widehat X_0)
\longrightarrow
H_1(\widehat X_L,\widehat X_0)
\longrightarrow
0.
\]
For ease of notation, set
\[
\mathfrak B'
:=
H_1(Y_C,\widehat X_0)\oplus H_1(Y_P,\widehat X_0).
\]
By Corollary~\ref{cor:orders-short-exact-sequence}, for every
$k_1,k_2\geq0$ with $k_1+k_2=k+3$, we have
\[
\Delta_{k+3}(\mathfrak B')
\;\Bigm|\;
\Delta_{k_1}\bigl(\operatorname{im}(f_1^{\mathrm{rel}})\bigr)
\,
\Delta_{k_2}\bigl(H_1(\widehat X_L,\widehat X_0)\bigr).
\]
We first observe that
\[
\Delta_2\bigl(\operatorname{im}(f_1^{\mathrm{rel}})\bigr)=1.
\]
Indeed, the long exact sequence of the pair
$(\widehat\Sigma,\widehat X_0)$ contains
\[
H_1(\widehat X_0)
\longrightarrow
H_1(\widehat\Sigma)
\longrightarrow
H_1(\widehat\Sigma,\widehat X_0)
\longrightarrow
H_0(\widehat X_0)
\longrightarrow
H_0(\widehat\Sigma).
\]
Since $\widehat X_0$ is discrete,
\[
H_1(\widehat X_0)=0.
\]
Moreover, by Remark~\ref{rem:H0-lifted-pieces-zero-zero}, the map
\[
H_0(\widehat X_0)
\longrightarrow
H_0(\widehat\Sigma)
\]
is an isomorphism. Hence
\[
H_1(\widehat\Sigma)
\cong
H_1(\widehat\Sigma,\widehat X_0).
\]
By Proposition~\ref{prop:homology-of-lifted-pieces-zero-zero},
\[
H_1(\widehat\Sigma)\cong R^2,
\]
so $\operatorname{im}(f_1^{\mathrm{rel}})$ is generated by at most two
elements. Therefore,
\[
\Delta_2\bigl(\operatorname{im}(f_1^{\mathrm{rel}})\bigr)=1.
\]
Taking $k_1=2$ in the preceding divisibility and using
Proposition~\ref{prop:relative-Alexander-module}, we obtain
\[
\Delta_{k+3}(\mathfrak B')
\;\Bigm|\;
\Delta_{k+1}\bigl(H_1(\widehat X_L,\widehat X_0)\bigr)
\doteq
\Delta_{L,k}.
\]
For the two summands of $\mathfrak B'$, the argument of
Proposition~\ref{prop:relative-Alexander-module} gives
\[
\Delta_i\bigl(H_1(Y_C,\widehat X_0)\bigr)
\doteq
\Delta_{i-1}\bigl(H_1(Y_C)\bigr)
\]
and
\[
\Delta_j\bigl(H_1(Y_P,\widehat X_0)\bigr)
\doteq
\Delta_{j-1}\bigl(H_1(Y_P)\bigr).
\]
The direct-sum formula for orders therefore gives
\[
\Delta_{k+3}(\mathfrak B')
\doteq
\gcd_{\substack{\\[2pt] i+j=k+1}}
\!\Bigl(
\Delta_i\bigl(H_1(Y_C)\bigr)
\,
\Delta_j\bigl(H_1(Y_P)\bigr)
\Bigr).
\]
Finally, by Proposition~\ref{prop:homology-of-lifted-pieces-zero-zero},
Remark~\ref{rem:specialization-zero-zero}, and
Corollary~\ref{cor:traldi-quotient-bounds}, we have
\[
D_{C,i}(\mathbf t_C)
\;\Bigm|\;
\Delta_i\bigl(H_1(Y_C)\bigr)
\qquad\text{and}\qquad
D_{P,j}(\mathbf t_1)
\;\Bigm|\;
\Delta_j\bigl(H_1(Y_P)\bigr).
\]
Consequently,
\[
\gcd_{\substack{\\[2pt] i+j=k+1}}
\!\Bigl(
D_{C,i}(\mathbf t_C)\,D_{P,j}(\mathbf t_1)
\Bigr)
\;\Bigm|\;
\Delta_{L,k},
\]
as required.
\end{proof}

We are now ready to complete the proof of
Theorem~\ref{thm:higher-single-component-case4}.

\begin{proof}[Proof of Theorem~\ref{thm:higher-single-component-case4}]
The result follows by combining
Corollary~\ref{cor:first-divisibility-zero-zero} and
Proposition~\ref{prop:reverse-divisibility-zero-zero}.
\end{proof}

\subsection{Winding or Linking Number Zero}
\label{subsec:winding-or-linking-zero}\hfill

\noindent
Throughout this subsection, we assume that
\[
\mathbf w_1=\mathbf 0
\qquad\text{and}\qquad
\boldsymbol\ell_1\neq\mathbf 0.
\]
Our main goal in this subsection is to prove
Theorem~\ref{thm:higher-single-component-case2}.

The argument follows the same general strategy as in the previous two
subsections. We continue to use the notation and framework established
above, although several of the spaces and modules appearing in the
argument take a different form in the present case. The first important
difference concerns the lift of the gluing torus.

\begin{defn}
\label{defn:cylinder-cover}
Let $\gamma_m$ be a meridian of $\Sigma=\partial N(C_1)$ and let
$\gamma_\ell$ be the longitude determined by the $0$-framing, both
based at $x_0$. Thus
\[
\pi_1(\Sigma,x_0)\cong\mathbb Z^2
\]
is generated by $[\gamma_m]$ and $[\gamma_\ell]$. Define
\[
\overline p_\Sigma\colon
\overline\Sigma\longrightarrow\Sigma
\]
to be the covering corresponding to the kernel of
\[
g_m\colon
\pi_1(\Sigma,x_0)\longrightarrow\mathbb Z,
\]
where
\[
g_m([\gamma_m])=0
\qquad\text{and}\qquad
g_m([\gamma_\ell])=1.
\]
\end{defn}

We can now formulate the analogue of
Lemmas~\ref{lem:lifted-pieces-nonzero} and
\ref{lem:lifted-pieces-zero-zero}. As in Remark~\ref{rem:indexing-sets-zero-zero}, the indexing sets of
Definition~\ref{defn:component-indexing-sets} simplify according to the
vanishing assumptions in the present case. With these indexing sets, we
have the following description of the lifted pieces.

\begin{lem}
\label{lem:lifted-pieces-winding-zero}
There are identifications
\begin{align}
Y_C
&=
\bigsqcup_{\mathbf n\in\Gamma_C}
\mathbf t_1^{\mathbf n} Z_C,
\label{eq:YC-components-winding-zero}
\\
Y_P
&=
\bigsqcup_{[\mathbf r]\in\Gamma_P}
\mathbf t_C^{\mathbf r}\widehat X_P,
\label{eq:YP-components-winding-zero}
\\
\widehat\Sigma
&=
\bigsqcup_{[(\mathbf n,\mathbf r)]\in\Gamma_\Sigma}
\mathbf t_1^{\mathbf n}\mathbf t_C^{\mathbf r}\overline\Sigma.
\label{eq:Sigma-components-winding-zero}
\end{align}
\end{lem}

\begin{proof}
The argument is identical to that of
Lemma~\ref{lem:lifted-pieces-nonzero} and
Lemma~\ref{lem:lifted-pieces-zero-zero}.
\end{proof}

Again, we start by analysing the first homologies of the lifted pieces.

\begin{lem}
\label{lem:homology-of-lift-of-torus}
There is a natural isomorphism of $R$-modules
\[
H_1(\widehat\Sigma)
\cong
R/\langle A_1-1\rangle.
\]
\end{lem}

\begin{proof}
By Definition~\ref{defn:cylinder-cover}, the covering
$\overline\Sigma$ is homeomorphic to an infinite cylinder. Its first
homology is generated by a lift of the meridian $\gamma_m$, while the
deck transformation corresponding to the longitude acts by translation.
Hence
\[
H_1(\overline\Sigma)
\cong
\mathbb Z[z^{\pm1}]/\langle z-1\rangle.
\]
As in the proof of
Proposition~\ref{prop:homology-of-lifted-pieces}, the decomposition in
Lemma~\ref{lem:lifted-pieces-winding-zero} gives
\[
H_1(\widehat\Sigma)
\cong
H_1(\overline\Sigma)
\otimes_{\mathbb Z[z^{\pm1}]}R.
\]
The coefficient homomorphism
\[
\mathbb Z[z^{\pm1}]
\longrightarrow
R
\]
is induced by the inclusion
$H_1(\Sigma)\to H_1(X_L)$ and sends
\[
z\longmapsto A_1.
\]
Therefore,
\[
H_1(\widehat\Sigma)
\cong
R/\langle A_1-1\rangle.
\]
\end{proof}

\begin{prop}
\label{prop:homology-of-lifted-pieces-0}
There are natural isomorphisms of $R$-modules
\begin{align}
H_1(Y_C)
&\cong
H_1(Z_C)\otimes_{R_C}R,
\label{eq:H1-YC-winding-zero}
\\
H_1(Y_P)
&\cong
H_1(\widehat X_P)\otimes_{R_P}R.
\label{eq:H1-YP-winding-zero}
\end{align}
\end{prop}

\begin{proof}
The argument is identical to that of
Proposition~\ref{prop:homology-of-lifted-pieces}, using the
description of the lifted pieces in
Lemma~\ref{lem:lifted-pieces-winding-zero}.
\end{proof}

As in Remark~\ref{rem:H0-lifted-pieces-zero-zero}, the description of
the connected components determines the zeroth homology of the lifted
pieces. We record only the facts that will be used below.

\begin{rem}
\label{rem:H0-lifted-pieces-winding-zero}
In the present case,
\[
H_0(\widehat\Sigma)
\cong
R/\langle A_1-1\rangle.
\]
Moreover, the argument of
Remark~\ref{rem:relative-end-terms-zero-zero} still applies, and hence
\[
H_0(\widehat\Sigma,\widehat X_0)=0.
\]
\end{rem}

We use the same construction and notation as in
Construction~\ref{cons:zero-zero-short-exact-sequences}. In particular,
we use the absolute and relative short exact sequences
\eqref{eq:00-case-MV} and \eqref{eq:00-case-MV-relative},
respectively. 

We first analyze the module $\mathfrak K$ appearing in
\eqref{eq:00-case-MV}.

\begin{prop}
\label{prop:module-k-0-case}
The module $\mathfrak K$ satisfies
\[
\Delta_0(\mathfrak K)
\;\Bigm|\;
A_1-1
\]
and
\[
\Delta_i(\mathfrak K)=1
\qquad\text{for all }i\geq1.
\]
\end{prop}

\begin{proof}
By Proposition~\ref{prop:generators-for-K}, the module $\mathfrak K$
is generated by
\[
g_{ij}=(t_{1i}-1)(t_j-1),
\qquad
1\leq i\leq m,\quad 2\leq j\leq\mu.
\]
As in the proof of
Proposition~\ref{prop:module-k-for-00-case}, let
\[
\{x_{ij}\}
\]
denote the corresponding generators of a free module mapping onto
$\mathfrak K$, and enlarge a finite generating set for the kernel by
the relations $z_{i,jk}$ and $z_{ik,j}$ introduced there.

Recall that $\mathfrak K$ is a submodule of $H_0(\widehat\Sigma)$.
Hence, by Remark~\ref{rem:H0-lifted-pieces-winding-zero},
\[
(A_1-1)\mathfrak K=0.
\]
We may therefore add, for every $i,j$, a relation $z_{ij}$ satisfying
\[
\widetilde\phi(z_{ij})
=
(A_1-1)x_{ij}.
\]
Thus we obtain a presentation
\[
R^{|I|+m\binom{\mu-1}{2}
+\binom{m}{2}(\mu-1)+m(\mu-1)}
\xrightarrow{\ \widetilde\phi\ }
R^{m(\mu-1)}
\longrightarrow
\mathfrak K
\longrightarrow
0.
\]

The two minors used in the proof of
Proposition~\ref{prop:module-k-for-00-case} remain minors of this
presentation matrix. Since their determinants are relatively prime, the
same argument gives
\[
\Delta_1(\mathfrak K)=1.
\]
It follows that
\[
\Delta_i(\mathfrak K)=1
\qquad\text{for all }i\geq1.
\]

It remains to consider $\Delta_0(\mathfrak K)$. Starting with the
matrix $\widetilde M_{11,2}$ from the proof of
Proposition~\ref{prop:module-k-for-00-case}, add a row corresponding
to the relation
\[
(A_1-1)x_{12}=0.
\]
With the ordering used there, the resulting square matrix is lower
triangular, and its determinant is
\[
(A_1-1)
(t_{11}-1)^{m-1}
(t_2-1)^{m(\mu-2)}.
\]
Similarly, starting with $\widetilde M_{12,3}$ and adding the row
corresponding to
\[
(A_1-1)x_{23}=0
\]
gives a square minor with determinant
\[
\pm
(A_1-1)
(t_{12}-1)^{m-1}
(t_3-1)^{m(\mu-2)}.
\]
As in the previous proof, when $m=1$ or $\mu=2$, the construction is
understood with any unavailable exchange omitted. The greatest common
divisor of these two determinants is $A_1-1$. Hence
\[
\Delta_0(\mathfrak K)
\;\Bigm|\;
A_1-1,
\]
as required.
\end{proof}

We next analyze the left-hand term in the short exact
sequence~\eqref{eq:00-case-MV}. The following is the analogue of
Lemma~\ref{lem:quotient-module-zero-zero} for the present case.

\begin{lem}
\label{lem:quotient-module-0-case}
For every $k\geq0$,
\[
\Delta_k(\mathfrak B)
\;\Bigm|\;
\gcd_{\substack{\\[2pt] i+j=k}}
\!\Bigl(
\Delta_{C',i}(\mathbf t_C)
\;
\Delta_{P',j}(\mathbf t_1,A_1)
\Bigr).
\]
\end{lem}

\begin{proof}
The argument is the one used in
Lemma~\ref{lem:quotient-module-zero-zero}, with one simplification.
By Lemma~\ref{lem:homology-of-lift-of-torus},
$H_1(\widehat\Sigma)$ is cyclic, generated by the lift of the meridian.
Thus a presentation matrix for $\mathfrak B$ has the form
\[
M_{\mathfrak B}
=
\begin{pmatrix}
M_C & 0 \\
r_C & r_P \\
0 & M_P
\end{pmatrix}.
\]
The determinant argument from
Lemma~\ref{lem:quotient-module-zero-zero} therefore gives
\[
E_k\left(
H_1(Y_C)/\operatorname{im}(f_C)
\oplus H_1(Y_P)
\right)
\subseteq
E_k(\mathfrak B),
\]
where $f_C$ denotes the projection of $f_1$ to
$H_1(Y_C)$.

Combining Lemma~\ref{lem:quotient-by-torus-image} with
Proposition~\ref{prop:homology-of-lifted-pieces-0}, and extending
scalars to $R$, gives
\[
\Delta_k(\mathfrak B)
\;\Bigm|\;
\gcd_{\substack{\\[2pt] i+j=k}}
\!\Bigl(
\Delta_{C',i}(\mathbf t_C)
\;
\Delta_{P',j}(\mathbf t_1,A_1)
\Bigr),
\]
as required.
\end{proof}

The preceding analysis gives the first divisibility in
Theorem~\ref{thm:higher-single-component-case2}.

\begin{cor}
\label{coro:first-divisibility-winding-zero}
For every $k\geq0$,
\[
\Delta_{L,k}
\;\Bigm|\;
\gcd_{\substack{\\[2pt] i+j=k}}
\!\Bigl(
\Delta_{C',i-1}(\mathbf t_C)
\;
\Delta_{P',j}(\mathbf t_1,A_1),
\;
(A_1-1)\Delta_{C',i}(\mathbf t_C)
\;
\Delta_{P',j}(\mathbf t_1,A_1)
\Bigr).
\]
\end{cor}

\begin{proof}
Applying Corollary~\ref{cor:orders-short-exact-sequence} to the short
exact sequence~\eqref{eq:00-case-MV} gives
\[
\Delta_{L,k}
\;\Bigm|\;
\gcd_{\substack{\\[2pt] k_1+k_2=k}}
\!\Bigl(
\Delta_{k_1}(\mathfrak B)\,
\Delta_{k_2}(\mathfrak K)
\Bigr).
\]
By Proposition~\ref{prop:module-k-0-case},
\[
\Delta_0(\mathfrak K)
\;\Bigm|\;
A_1-1,
\qquad
\Delta_{k_2}(\mathfrak K)=1
\quad\text{for all }k_2\geq1.
\]
Consequently, considering $k_2=0$ and $k_2=1$, we obtain
\[
\Delta_{L,k}
\;\Bigm|\;
\gcd\Bigl(
\Delta_{k-1}(\mathfrak B),
(A_1-1)\Delta_k(\mathfrak B)
\Bigr).
\]
Applying Lemma~\ref{lem:quotient-module-0-case} to the two terms on the
right gives the desired divisibility.
\end{proof}

We now turn to the relative Mayer--Vietoris sequence
\eqref{eq:relative-MV-sequence}, and in particular to the short exact
sequence~\eqref{eq:00-case-MV-relative} obtained from it. The main new
feature in the present case is the structure of
$H_1(\widehat\Sigma,\widehat X_0)$.

\begin{prop}
\label{prop:relative-sequence-first-term-0-case}
There is a natural isomorphism of $R$-modules
\[
H_1(\widehat\Sigma,\widehat X_0)
\cong
R\oplus R/\langle A_1-1\rangle.
\]
In particular,
$\operatorname{im}(f_1^{\mathrm{rel}})$ is generated by at most two
elements.
\end{prop}

\begin{proof}
Arguing as in the proof of
Proposition~\ref{prop:reverse-divisibility-zero-zero}, the long exact
sequence of the pair $(\widehat\Sigma,\widehat X_0)$, together with
Lemma~\ref{lem:homology-of-lift-of-torus} and Remark~\ref{rem:H0-lifted-pieces-winding-zero}, gives a short exact
sequence
\[
0
\longrightarrow
R/\langle A_1-1\rangle
\longrightarrow
H_1(\widehat\Sigma,\widehat X_0)
\longrightarrow
\langle A_1-1\rangle
\longrightarrow
0.
\]
Since $\langle A_1-1\rangle\cong R$, the sequence splits, and therefore
\[
H_1(\widehat\Sigma,\widehat X_0)
\cong
R\oplus R/\langle A_1-1\rangle.
\]
The final assertion follows immediately.
\end{proof}

\begin{lem}
\label{lem:im-in-relative-case-0-case}
The module $\operatorname{im}(f_1^{\mathrm{rel}})$ satisfies
\[
\Delta_1\bigl(\operatorname{im}(f_1^{\mathrm{rel}})\bigr)
\;\Bigm|\;
A_1-1
\]
and
\[
\Delta_i\bigl(\operatorname{im}(f_1^{\mathrm{rel}})\bigr)=1
\qquad\text{for all }i\geq2.
\]
\end{lem}

\begin{proof}
By Proposition~\ref{prop:relative-sequence-first-term-0-case},
$\operatorname{im}(f_1^{\mathrm{rel}})$ is generated by at most two
elements. Hence
\[
\Delta_i\bigl(\operatorname{im}(f_1^{\mathrm{rel}})\bigr)=1
\qquad\text{for all }i\geq2.
\]
Moreover, the image of the summand
\[
R/\langle A_1-1\rangle
\subseteq
H_1(\widehat\Sigma,\widehat X_0)
\]
is annihilated by $A_1-1$. Thus a two-generator presentation may be
chosen with a relation whose only nonzero entry is $A_1-1$. Therefore
\[
A_1-1
\in
E_1\bigl(\operatorname{im}(f_1^{\mathrm{rel}})\bigr),
\]
and hence
\[
\Delta_1\bigl(\operatorname{im}(f_1^{\mathrm{rel}})\bigr)
\;\Bigm|\;
A_1-1.
\]
\end{proof}

As in Proposition~\ref{prop:reverse-divisibility-zero-zero}, set
\[ \mathfrak B' := H_1(Y_C,\widehat X_0)\oplus H_1(Y_P,\widehat X_0). \] 
We next record the lower bound on its orders.

\begin{lem}
\label{lem:middle-term-in-relative-ses-0-case}
For every $k\geq0$,
\[
\gcd_{\substack{\\[2pt] i+j=k-2}}
\!\Bigl(
D_{C,i}(\mathbf t_C)
\;
\Delta_{P',j}(\mathbf t_1,A_1)
\Bigr)
\;\Bigm|\;
\Delta_k(\mathfrak B').
\]
\end{lem}

\begin{proof}
As in the proof of
Proposition~\ref{prop:reverse-divisibility-zero-zero}, the relative
Alexander-module shift and the direct-sum formula give
\[
\Delta_k(\mathfrak B')
\doteq
\gcd_{\substack{\\[2pt] i+j=k-2}}
\!\Bigl(
\Delta_i\bigl(H_1(Y_C)\bigr)
\;
\Delta_j\bigl(H_1(Y_P)\bigr)
\Bigr).
\]
By Proposition~\ref{prop:homology-of-lifted-pieces-0} and
Corollary~\ref{cor:traldi-quotient-bounds},
\[
D_{C,i}(\mathbf t_C)
\;\Bigm|\;
\Delta_i\bigl(H_1(Y_C)\bigr),
\]
whereas
\[
\Delta_j\bigl(H_1(Y_P)\bigr)
\doteq
\Delta_{P',j}(\mathbf t_1,A_1).
\]
The stated divisibility follows.
\end{proof}

The preceding computations give the two opposite divisibilities.

\begin{cor}
\label{coro:reverse-divisibilities-winding-zero}
For every $k\geq0$,
\[
\gcd_{\substack{\\[2pt] i+j=k+1}}
\!\Bigl(
D_{C,i}(\mathbf t_C)
\;
\Delta_{P',j}(\mathbf t_1,A_1)
\Bigr)
\;\Bigm|\;
\Delta_{L,k},
\]
and
\[
\gcd_{\substack{\\[2pt] i+j=k}}
\!\Bigl(
D_{C,i}(\mathbf t_C)
\;
\Delta_{P',j}(\mathbf t_1,A_1)
\Bigr)
\;\Bigm|\;
(A_1-1)\Delta_{L,k}.
\]
\end{cor}

\begin{proof}
Applying Corollary~\ref{cor:orders-short-exact-sequence} to the relative
short exact sequence~\eqref{eq:00-case-MV-relative}, we have
\[
\Delta_n(\mathfrak B')
\;\Bigm|\;
\Delta_{k_1}\bigl(\operatorname{im}(f_1^{\mathrm{rel}})\bigr)
\,
\Delta_{k_2}\bigl(H_1(\widehat X_L,\widehat X_0)\bigr)
\]
whenever $k_1+k_2=n$.

Taking
\[
(n,k_1,k_2)=(k+3,2,k+1)
\]
and using
Lemma~\ref{lem:im-in-relative-case-0-case} and
Proposition~\ref{prop:relative-Alexander-module} gives
\[
\Delta_{k+3}(\mathfrak B')
\;\Bigm|\;
\Delta_{L,k}.
\]
Similarly, taking
\[
(n,k_1,k_2)=(k+2,1,k+1)
\]
gives
\[
\Delta_{k+2}(\mathfrak B')
\;\Bigm|\;
(A_1-1)\Delta_{L,k}.
\]
The two stated divisibilities now follow from
Lemma~\ref{lem:middle-term-in-relative-ses-0-case}.
\end{proof}

We are now ready to complete the proof of
Theorem~\ref{thm:higher-single-component-case2}.

\begin{proof}[Proof of Theorem~\ref{thm:higher-single-component-case2}]
The three divisibility statements follow from
Corollary~\ref{coro:first-divisibility-winding-zero} and
Corollary~\ref{coro:reverse-divisibilities-winding-zero}.
\end{proof}

\section{Examples}
\label{sec:Examples}

We now give several examples illustrating the results of this paper. The first shows that one of lower bounds from Theorem~\ref{thm:higher-single-component-case2} can be realised.

\begin{exa}
\label{ex:lower-bound-case2}
Our first example examines the bounds in
Theorem~\ref{thm:higher-single-component-case2}. The augmented pattern
$P'$, the companion link $C$, and the resulting satellite link $L$ are
shown in Figure~\ref{fig:example-lower-bound-case2}.

\begin{figure}[h]
    \centering
    \includegraphics[width=0.9\textwidth]{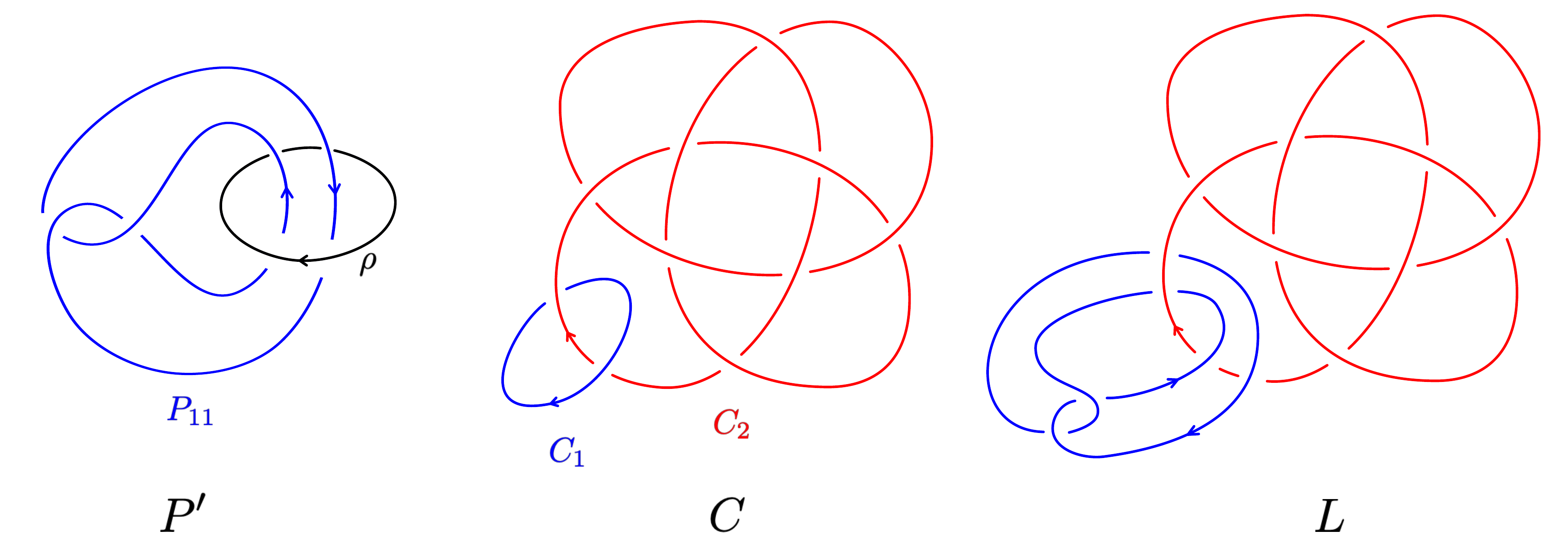}
    \caption{Example 1: The augmented pattern $P'$, the companion link
    $C$, and the resulting satellite link $L$.}
    \label{fig:example-lower-bound-case2}
\end{figure}

In this example,
\[
\mathbf w_1=\mathbf 0
\qquad\text{and}\qquad
\boldsymbol\ell_1\neq\mathbf 0,
\]
so Theorem~\ref{thm:higher-single-component-case2} applies. Moreover,
\[
A_1=t_2,
\]
and the relevant Alexander polynomials of the augmented pattern are
\[
\Delta_{P',0}(\mathbf t_1,A_1)
=
(t_{11}-1)(t_2-1),
\qquad
\Delta_{P',1}(\mathbf t_1,A_1)=1.
\]
For the sublink $C'=C\setminus C_1$, we have
\[
\Delta_{C',0}(\mathbf t_C)
=
(t_2^2-3t_2+1)(t_2^2-t_2+1)^2,
\]
\[
\Delta_{C',1}(\mathbf t_C)
=
t_2^2-t_2+1,
\qquad
\Delta_{C',2}(\mathbf t_C)=1.
\]
As a result, the quantities from Definition~\ref{defn:DC-DP} are:
\begin{align*}
D_{C,0}&=(t_2-1)\Delta_{C',0}=(t_2-1)(t_2^2-3t_2+1)(t_2^2-t_2+1)^2,
\\
D_{C,1}&=\gcd(\Delta_{C',0},(t_2-1)\Delta_{C',1})=t_2^2-t_2+1,
\\
D_{C,2}&=\gcd(\Delta_{C',1},(t_2-1)\Delta_{C',2})=1.
\end{align*}

We first consider the two lower bounds in
Theorem~\ref{thm:higher-single-component-case2} for $k=1$. The first
gives
\[
\gcd_{\substack{\\[2pt] i+j=2}}
\!\Bigl(
D_{C,i}(\mathbf t_C)
\;
\Delta_{P',j}(\mathbf t_1,A_1)
\Bigr)
\;\Bigm|\;
\Delta_{L,1}.
\]
In particular, the terms corresponding to $(i,j)=(0,2),(1,1)$ and $(2,0)$
are
\[
D_{C,0}\Delta_{P',2}=(t_2-1)(t_2^2-3t_2+1)(t_2^2-t_2+1)^2
\]
\[
D_{C,1}\Delta_{P',1}
=
t_2^2-t_2+1
\]
and
\[
D_{C,2}\Delta_{P',0}
=
(t_{11}-1)(t_2-1),
\]
respectively. Since
\[
\gcd\!\left(
t_2^2-t_2+1,\,
(t_{11}-1)(t_2-1)
\right)
=1,
\]
the first lower bound is trivial.

The second lower bound gives
\[
\gcd_{\substack{\\[2pt] i+j=1}}
\!\Bigl(
D_{C,i}(\mathbf t_C)
\;
\Delta_{P',j}(\mathbf t_1,A_1)
\Bigr)
\;\Bigm|\;
(A_1-1)\Delta_{L,1}.
\]
The two terms in this greatest common divisor are
\[
D_{C,0}\Delta_{P',1}
=
(t_2-1)(t_2^2-3t_2+1)(t_2^2-t_2+1)^2
\]
and
\[
D_{C,1}\Delta_{P',0}
=
(t_2^2-t_2+1)(t_{11}-1)(t_2-1).
\]
Hence
\[
\gcd\ \!\Bigl(
D_{C,0}\Delta_{P',1},
D_{C,1}\Delta_{P',0}
\Bigr)
=
(t_2-1)(t_2^2-t_2+1).
\]
Since $A_1=t_2$, the second lower bound therefore gives
\[
(t_2-1)(t_2^2-t_2+1)
\;\Bigm|\;
(t_2-1)\Delta_{L,1}.
\]
it follows that
\[
t_2^2-t_2+1
\;\Bigm|\;
\Delta_{L,1}.
\]

We now consider the upper bound in
Theorem~\ref{thm:higher-single-component-case2}. For $k=1$, it gives
\[
\Delta_{L,1}
\;\Bigm|\;
\gcd_{\substack{\\[2pt] i+j=1}}
\!\Bigl(
\Delta_{C',i-1}(\mathbf t_C)
\;
\Delta_{P',j}(\mathbf t_1,A_1),
\;
(A_1-1)\Delta_{C',i}(\mathbf t_C)
\;
\Delta_{P',j}(\mathbf t_1,A_1)
\Bigr).
\]
Using the convention that terms involving Alexander polynomials with
negative indices are omitted, the relevant terms are
\[
(A_1-1)\Delta_{C',0}\Delta_{P',1}
=
(t_2-1)\Delta_{C',0},
\]
\[
(A_1-1)\Delta_{C',1}\Delta_{P',0}
=
(t_{11}-1)(t_2-1)^2(t_2^2-t_2+1)
\]
and
\[
\Delta_{C',0}\Delta_{P',0}
=
(t_{11}-1)(t_2-1)\Delta_{C',0}.
\]
Hence, the greatest common divisor of these three terms is
\[
(t_2-1)(t_2^2-t_2+1).
\]
Thus, the upper bound gives
\[
\Delta_{L,1}
\;\Bigm|\;
(t_2-1)(t_2^2-t_2+1).
\]

A direct computation gives
\[
\Delta_{L,1}
=
t_2^2-t_2+1.
\]
Thus, in this example the first lower bound is trivial, while the
second lower bound is sharp and determines $\Delta_{L,1}$ up to a
unit. The upper bound, on the other hand, differs from
$\Delta_{L,1}$ by the additional factor $t_2-1$.
\end{exa}

To contrast the above example we now give a second example where neither of the lower bounds from Theorem~\ref{thm:higher-single-component-case2} realise $\Delta_{L,1}$.

\begin{exa}
In this example we take the same augmented pattern $P'$ from the previous example, and then the companion $C$ and resulting satellite $L$ as shown in Figure~\ref{fig:example-lower-bound-case2}.

\begin{figure}[h]
    \centering
    \includegraphics[width=0.8\textwidth]{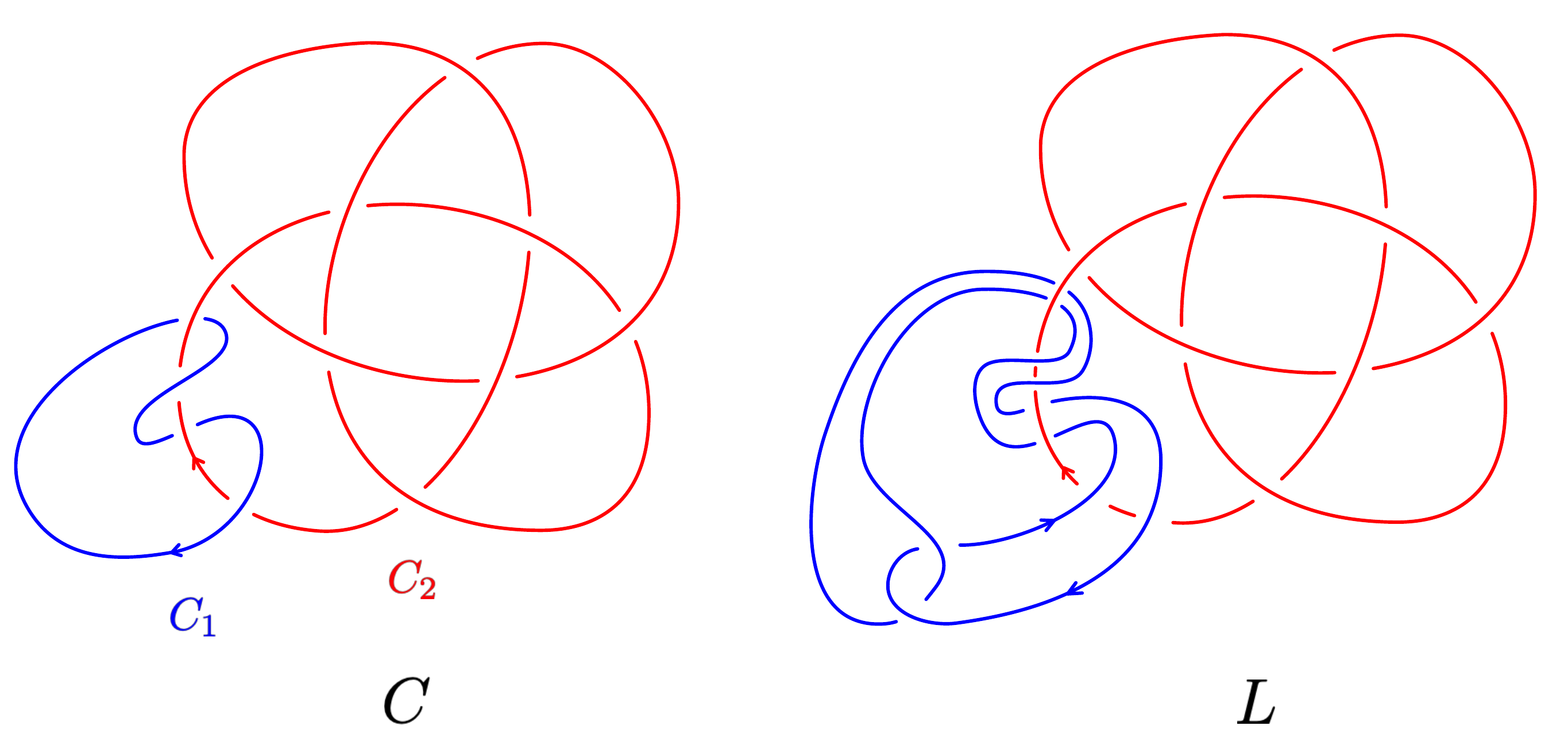}
    \caption{Example 2: The companion link
    $C$ and the resulting satellite link $L$.}
    \label{fig:example-lower-bound-case2}
\end{figure}

Again in this example we have,
\[
\mathbf w_1=\mathbf 0
\qquad\text{and}\qquad
\boldsymbol\ell_1\neq\mathbf 0,
\]
so Theorem~\ref{thm:higher-single-component-case2} applies. As $C'=C\backslash C_1$ is the same knot as in the previous example we have the use those computations. That is we have,
\[
D_{C,0}=(t_2-1)(t_2^2-3t_2+1)(t_2^2-t_2+1)^2, \qquad D_{C,1}=t_2^2-t_2+1, \qquad D_{C,2}=1.
\]
Similarly, as the pattern is the same we have,
\[
\Delta_{P',0}(\mathbf t_1,A_1)
=
(t_{11}-1)(t_2^2-1),
\qquad
\Delta_{P',1}(\mathbf t_1,A_1)=1.
\]

Note that $A_1=t_2^2$. So applying the lower bounds from Theorem~\ref{thm:higher-single-component-case2} using $k=1$ and the computations as in the previous example we have,
\[
1 \;\Bigm|\; \Delta_{L,1}
\qquad
\text{and}
\qquad
(t_2-1)(t_2^2-t_2+1) \;\Bigm|\; (t_2^2-1)\Delta_{L,1}.
\]
For the upper bound the relevant terms are,
\[
(A_1-1)\Delta_{C',0}\Delta_{P',1}
=
(t_2^2-1)\Delta_{C',0},
\]
\[
(A_1-1)\Delta_{C',1}\Delta_{P',0}
=
(t_{11}-1)(t_2^2-1)^2(t_2^2-t_2+1)
\]
and
\[
\Delta_{C',0}\Delta_{P',0}
=
(t_{11}-1)(t_2^2-1)\Delta_{C',0}.
\]
Hence, the upper bound gives,
\[
\Delta_{L,1} \,\bigm|\, (t_2^2-1)(t_2^2-t_2+1).
\]

Finally, a direct computation gives that,
\[
\Delta_{L,1}=t_2^2-t_2+1.
\]
Thus, in this example neither of the lower bounds are sharp. The first is trivial and the second is off by a factor of $(t_2+1)$. Further, again the upper bound fails to be sharp by a factor of $(A_1-1)$.
\end{exa}

\bibliographystyle{alpha} 
\bibliography{bibtemplate}

@article{Seifert1950,
author = {Seifert, Herbert and Threlfall, William},
address = {Cambridge, UK},
copyright = {Copyright © Canadian Mathematical Society 1950},
issn = {0008-414X},
journal = {Canadian journal of mathematics},
language = {eng},
number = {1},
pages = {1-15},
publisher = {Cambridge University Press},
title = {{O}ld and {N}ew {R}esults on {K}nots},
volume = {2},
year = {1950},
}

@article{Torres,
  title={{On the Alexander Polynomial}},
  author={Guillermo del Castillo Torres},
  journal={Annals of Mathematics},
  year={1953},
  volume={57},
  pages={57},
  url={https://api.semanticscholar.org/CorpusID:123894744}
}

@book{Lickorish,
author = {Lickorish, W. B. Raymond},
address = {New York, NY},
booktitle = {An Introduction to Knot Theory},
edition = {1st ed. 1997.},
isbn = {1-4612-6869-9},
language = {eng},
publisher = {Springer New York},
series = {Graduate Texts in Mathematics, 175},
title = {An Introduction to Knot Theory },
year = {1997},
}

@article{Petkova,
author = {Petkova, Ina},
address = {SOMERVILLE},
copyright = {Copyright 2018 Elsevier B.V., All rights reserved.},
issn = {1527-5256},
journal = {Journal of symplectic geometry},
language = {eng},
number = {1},
pages = {227-277},
publisher = {Int Press Boston, Inc},
title = {{The decategorification of bordered Heegaard Floer homology}},
volume = {16},
year = {2018},
}

@InProceedings{Livingston,
author="Livingston, Charles
and Melvin, Paul",
title="{Abelian invariants of satellite knots}",
booktitle="Geometry and Topology",
year="1985",
publisher="Springer Berlin Heidelberg",
address="Berlin, Heidelberg",
pages="217--227",
isbn="978-3-540-39738-0"
}

@book{Burde,
  author = {Burde, Gerhard and Zieschang, Heiner and Heusener, Michael},
  title = {Knots},
  series = {De Gruyter Studies in Mathematics},
  volume = {5},
  pages = {xiv+417},
  publisher = {De Gruyter, Berlin},
  edition = {extended},
  year = {2014},
  isbn = {978-3-11-027074-7; 978-3-11-027078-5},
  mrclass = {57-01 (57M25)},
  mrnumber = {3156509},
  mrreviewer = {Swatee\ Naik},
}

@book{EisenbudNeumann,
author = {Eisenbud, David and Neumann, W. D.},
address = {Princeton},
booktitle = {Three-dimensional link theory and invariants of plane curve singularities},
isbn = {0691083800},
language = {eng},
lccn = {LC 85-545},
publisher = {Princeton University Press},
series = {Annals of mathematics studies ; no. 110},
title = {Three-dimensional link theory and invariants of plane curve singularities },
year = {1985},
}

@misc{Friedl,
  title={Topology},
  author={Friedl, Stefan},
  year={2023},
  publisher={Lecture notes},
  note={URL: \url{https://friedl.app.uni-regensburg.de/papers/1t-total-public-september-1.pdf}}
}

@misc{Signature,
author = {Degtyarev, Alex and Florens, Vincent and Lecuona, Ana G.},
copyright = {legaldeposit},
language = {eng},
publisher = {Oxford University Press},
title = {{The Signature of a Splice}},
year = {20170415},
}

@article{Cimasoni,
author = {Cimasoni, David},
address = {Cambridge, UK},
copyright = {Copyright © Edinburgh Mathematical Society 2005},
issn = {0013-0915},
journal = {Proceedings of the Edinburgh Mathematical Society},
language = {eng},
number = {1},
pages = {61-73},
publisher = {Cambridge University Press},
title = {{The Conway function of a splice}},
volume = {48},
year = {2005},
}

@misc{stacks-project,
  shorthand    = {Stacks},
  author       = {The {Stacks Project Authors}},
  title        = {\textit{Stacks Project}},
  howpublished = {\url{https://stacks.math.columbia.edu}},
  year         = {2018},
}

@article{Blanchfield,
author = {Blanchfield, Richard C.},
address = {PRINCETON},
issn = {0003-486X},
journal = {Annals of mathematics},
language = {eng},
number = {2},
pages = {340-356},
publisher = {Princeton University},
title = {{Intersection Theory of Manifolds With Operators with Applications to Knot Theory}},
volume = {65},
year = {1957},
}

@article{Traldi,
 ISSN = {00029947},
 URL = {http://www.jstor.org/stable/1998469},
 author = {Lorenzo Traldi},
 journal = {Transactions of the American Mathematical Society},
 number = {2},
 pages = {593--610},
 publisher = {American Mathematical Society},
 title = {{A Generalization of Torres' Second Relation}},
 urldate = {2026-09-03},
 volume = {269},
 year = {1982}
}
\end{document}